\documentclass[a4paper]{amsart} 
\usepackage{amsmath}
\usepackage{amssymb}
\usepackage{amsthm}
\usepackage{mathrsfs}
\usepackage{graphicx}
\usepackage{mathtools}
\usepackage[colorlinks,pdfdisplaydoctitle,linkcolor=blue,citecolor=blue]{hyperref}
\usepackage{caption}
\usepackage{subcaption}
\usepackage{url}
\usepackage{epstopdf}
\usepackage{pgf,tikz}
\usepackage{todonotes}
\usepackage{bbm}
\usepackage{extarrows}
\usetikzlibrary{arrows}
\usepackage{algorithm}
 \usepackage{algorithmicx}
 \usepackage{algpseudocode}

 \usepackage{xcolor,pifont}
\newcommand*\colourcheck[1]{%
  \expandafter\newcommand\csname #1check\endcsname{\textcolor{#1}{\ding{52}}}%
}
\colourcheck{blue}
\colourcheck{green}

\newcommand*\colourmark[1]{%
  \expandafter\newcommand\csname #1mark\endcsname{\textcolor{#1}{\ding{55}}}%
}
\colourmark{red}

\usepackage{amssymb,amsthm}    
\usepackage{graphicx}          
\usepackage{amsmath}         
\usepackage{caption}
\usepackage{subcaption}
\usepackage{dsfont}
\usepackage{listings}
\usepackage{bbm}
\usepackage{color}
\usepackage{ dsfont }
\usepackage{enumerate}
 \usepackage{relsize}
\usepackage{xfrac}

\usepackage{changes}
\newcommand{\stkout}[1]{\ifmmode\text{\sout{\ensuremath{#1}}}\else\sout{#1}\fi}
\setdeletedmarkup{\stkout{#1}}

\renewcommand{\Delta}{\triangle}

\usepackage{pifont}
\usepackage{hhline}

\definecolor{darkblue}{rgb}{0,0,0.7}

\definecolor{darkgreen}{rgb}{0.01,0.75,0.24}

\def \Ee[#1]{\mathcal{E}^{\text{{#1}}}}
\def\R{\mathbb{R}}

\def\pa[#1,#2]{\frac{\partial {#1}}{\partial {#2}} }

\def\idom[#1,#2,#3]{\int_{#1}\hspace{1pt} {#2} \hspace{1pt} \text{d}{#3}}
\def\res[#1,#2]{\left.{#1}\right|_{#2}}

\def\var[#1,#2]{\langle \delta \mathcal{E}^{\text{{#1}}}({#2}),v\rangle}
\def\vars[#1,#2,#3]{\langle \delta^2\mathcal{E}^{\text{{#1}}}({#2})v,{#3}\rangle}
\def\vard[#1,#2,#3,#4]{\langle \delta\mathcal{E}^{\text{{#1}}}({#2})-\delta\mathcal{E}^{\text{{#3}}}({#4}),v\rangle}

\def\E{\mathbb{E}}

\newcommand{\cF}{\mathcal{F}}

\newcommand{\calI}{\mathcal{I}}

\DeclareMathOperator*{\argmax}{arg\,max}
\DeclareMathOperator*{\argmin}{arg\,min}

\newcommand{\be}{\begin{equation}}
\newcommand{\en}{\end{equation}}
\newcommand{\ben}{\begin{equation*}}
\newcommand{\enn}{\end{equation*}}
\newcommand{\bea}{\begin{aligned}}
\newcommand{\ena}{\end{aligned}}

\def\ba#1\ena{\begin{align}#1\end{align}}
\def\ban#1\enan{\begin{align*}#1\end{align*}}

\theoremstyle{plain}
\newtheorem{thm}{Theorem}[section]

\newtheorem{assumption}[thm]{Assumption}

\newtheorem{proposition}[thm]{Proposition}

\newtheorem{remark}[thm]{Remark}

\numberwithin{equation}{section}

\begin{document}

\title[Adaptive Tikhonov strategies for stochastic EKI]{Adaptive Tikhonov strategies for stochastic ensemble Kalman inversion}

\author[S. Weissmann] {Simon Weissmann}
\address{Interdisciplinary Center for Scientific Computing, University of Heidelberg, 69120 Heidelberg, Germany}
\email{simon.weissmann@uni-heidelberg.de}

\author[N. K. Chada] {Neil K. Chada}
\address{Applied Mathematics and Computational Science Program, King Abdullah University of Science and Technology, Thuwal, 23955, KSA}
\email{neilchada123@gmail.com}

\author[C. Schillings] {Claudia Schillings}
\address{Mannheim School of Computer Science and Mathematics, University of Mannheim, 68131 Mannheim, Germany}
\email{c.schillings@uni-mannheim.de}

\author[X. T. Tong] {Xin T. Tong}
\address{Department of Mathematics, National University of Singapore, 119077, Singapore}
\email{mattxin@nus.edu.sg}



\begin{abstract} 
{
Ensemble Kalman inversion (EKI) is a derivative-free optimizer aimed at solving inverse problems, taking motivation from the celebrated ensemble Kalman filter. The purpose of this article is to consider the introduction of adaptive Tikhonov strategies for EKI. This work builds upon Tikhonov EKI (TEKI) which was proposed for a fixed regularization constant. By adaptively learning the regularization parameter, this procedure is known to improve the recovery of the underlying unknown. For the analysis, we consider a continuous-time setting where we extend known results such as well-posdeness and convergence of various loss functions, but with the addition of noisy observations.  Furthermore, we allow a time-varying noise and regularization covariance in our presented convergence result {which mimic adaptive regularization schemes.} In turn we present three adaptive regularization schemes, which are highlighted from both the deterministic and Bayesian approaches for inverse problems, which include bilevel optimization, the MAP formulation and covariance learning. We numerically test these schemes and the theory on linear and nonlinear partial differential equations, where they outperform the non-adaptive TEKI and EKI.}
\end{abstract}

\maketitle

\section{Introduction}
\label{sec:intro}

Inverse problems \cite{BB18,EHN96,KS04,AMS10} are ubiquitous in nature, science and engineering. Mathematically they are concerned with the recovery of some quantity of interest $u \in X$ from noisy measurements $y \in \R^{K}$, which is modeled by
\begin{equation}
\label{eq:inv}
y = G(u) + \eta, \quad \eta \sim {N}(0,\Gamma).
\end{equation}
In \eqref{eq:inv}, $G:X \rightarrow \R^K$ is a forward operator between the parameter and observation space, and $\eta$ is an additive Gaussian noise with known positive definite covariance $\Gamma\in\R^{K\times K}$. For simplicity, we will assume that the parameter space is  finite dimensional, in particular we will assume $X=\R^{d_u}$ where $d_u$ denotes the dimension of the unknown $u$. Inverse problems in the traditional, or classical, sense are aimed at minimizing some functional of the form \begin{equation}
\label{eq:func1}
\mathcal{I}(u;y) := \frac{1}{2} \big\| y - G(u)\big\|^2_{\Gamma} + \mathcal{S}_{\lambda}(u),
\end{equation}
where $\mathcal{S}_{\lambda}(u)$ is a penalty term to act as a regularizer with regularization parameter $\lambda>0$ and $\|\cdot\|_\Gamma = \|\Gamma^{-1/2}\cdot\|$ denotes a scaled euclidean norm in $\R^K$.  Common regularization schemes \cite{BB18,EHN96} include Tikhonov, $\mathcal{S}_{\lambda}(u) = \frac{\lambda}{2}\|u\|^2$,
 as well as $\ell_p$ regularization $\mathcal{S}_{\lambda}(u)=\lambda|u|_p^p$ and total variation, $\mathcal{S}_{\lambda}(u)=\lambda| \nabla u|^2$.

Recent advances in inverse problems are sparked by many new ideas from of both statistics and optimization literature. In particular, one methodology which has shown great promise is ensemble Kalman inversion (EKI). EKI can be viewed as the application of filtering \cite{BC09,LSZ15}, in particular the ensemble Kalman filter (EnKF) \cite{GE09,GE03}, to solve inverse problems. It is based on updating an ensemble of particles $\{u^{j}\}_{j=0}^J$, through the sample mean and sample covariances.  An attractive property associated with EKI is that it is derivative-free \cite{LMW19,NS17}, which can reduce the computational cost considerably for high-dimensional problems. 

{Our aim in this work is to extend the current results of Tikhonov regularization for EKI (abbreviated to TEKI), where we consider learning the regularization parameter adaptively, opposed to choosing a fixed parameter value for $\lambda$.  {On the one side, we present a theoretical convergence analysis which allows for time depending regularization and noise scaling, and on the other side,  for tuning the outcome of TEKI we are going to propose a row of adaptive regularization schemes based on bilevel optimization as well as on hierarchical Bayesian methods.} Before going into details, we provide a review of related literature of the ensemble Kalman inversion.}

\subsection{Literature overview}
{
Since the formulation of EKI \cite{ILS13,LR09} as a derivative-free optimizer, with respect to {inverse problems based on PDE-constrained optimization problems}, there has been a significant increase in related research. This can be divided into primarily two entities, one is the incorporation of Bayesian \& uncertainty quantification (UQ) methodologies, and the second is focused on optimization \& variational approaches.  }

While the EKI can approximate the posterior distribution under linear and Gaussian assumptions,  it is known to be not consistent with the posterior distribution for nonlinear forward maps \cite{ErnstEtAl2015}.  In order to quantify the approximation of posterior distributions the EKI has been analysed in its mean field limit connected to the Fokker--Planck equation \cite{DL19,MHGV2018}.  Other examples include Bayesian techniques that modify EKI so it can approximate posterior distributions \cite{DL21,GHLS19,HSS21,RW2021}.  Furthermore, hierarchical techniques have been included which seek to incorporate the learning of parametrized information on the underlying unknown. In the case where the unknown is a Gaussian random field, certain hyperparameters could be the length-scale or regularity of the covariance operator \cite{NKC18,CIRS18}. 

Much of the focus on EKI analysis has been with respect to its capabilities as an optimizer. Recent papers have adopted common optimization procedures such as the incorporation of box-constraints \cite{CSW19}, and the introduction of Tikhonov and iterative regularization \cite{CT19,CST19,MAI16,YL21,SSW20}. Most of the regularization schemes introduced so far have been for a fixed choice of regularization parameter $\lambda>0$. However, some recent work has considered adaptive choices and convergence \cite{IY20,PS21}, related to iterative regularization. As EKI updates an ensemble of particles through the sample mean and covariances, Tikhonov regularization is quite a natural choice as it can be interpreted as a form of Gaussian regularization.  Furthermore, the EKI has been shown to be a promising optimization method for the training task in different machine learning applications \cite{GSW2020,KS18}. Other work related to the EKI has been on deriving theory for both the continuous and discrete formulations \cite{BSW18,BSWW19,BSWW2021,bungert2021, CT19, SS17}. {However, most existing analyses assume either the observation model contains no noise, or the EKI iteration follows a deterministic formula. These assumptions are not practical when we include adaptive learning procedures.} {Therefore, our analysis will be based on the stochastic formulation of EKI without noise-free assumption on the underlying noise model, and we will allow time depending noise and regularization scales in order to mimic adaptive choices of the regularization scheme.  Some of the schemes we will propose are based on data-driven regularization, which has seen a recent interest in the inverse problem community. Such examples include through neural networks, projection and dictionary learning \cite{AMOS16,AKS20,LSA20,LOS18}.}

\subsection{Contributions of this work} 

\begin{figure}[h!]
\includegraphics[width=0.6\textwidth]{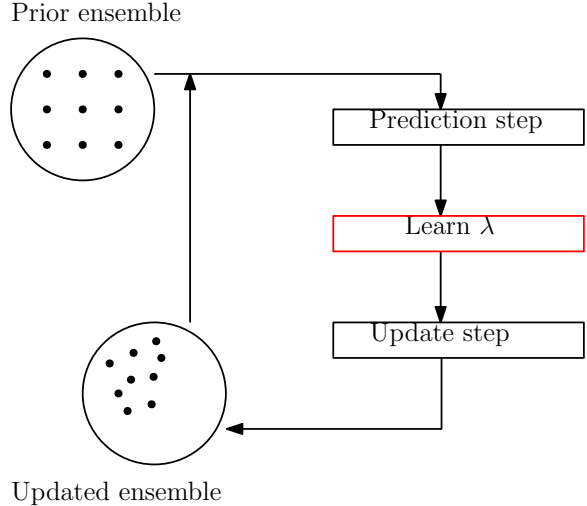}
\caption{Sketch representation of regularized ensemble Kalman inversion with the inclusion of data-driven learning.}
\label{fig:intro_reg}
\end{figure}

We develop adaptive strategies for Tikhonov EKI (TEKI), where we consider the task of choosing the regularization parameter within the iterative method of finding the underlying unknown parameter.  To have some theoretical understanding of our new algorithms, we extend the EKI analysis in the noisy regime of data. This is challenging, as noise in the observations leads to an unstable behavior of the original EKI method.  We will alleviate the stability issues by using ideas from Tikhonov regularization, which has been successfully applied to EKI in the noise-free case \cite{CST19}.  {Furthermore,  in our theoretical convergence analysis we consider time depending noise and regularization scales modeling the underyling adaptive regularization scheme. We introduce a number of new adaptive Tikhonov strategies aimed at improving stability within the noisy EKI, and outperforming TEKI for a fixed regularization parameter.  In Figure \ref{fig:intro_reg} we show the idea of learning the regularization parameter within EKI.  Viewing the prediction step as generation of training data, the adaptive regularization scheme comes right afterwards.  We make the following contributions:}

\smallskip
\begin{itemize}
 \item We extend the current work on TEKI for a fixed regularization parameter \cite{CST19}.  {Therefore, we consider the stochastic formulation of EKI viewed as coupled system of stochastic differential equations (SDE) resulting from the continuous time limit.  Namely with this we derive a continuous-time limit and in the linear setting we present a number of results such as well-posedness of the scheme, the ensemble collapse of the particles and convergence to the minimizer of the functional. To help achieve the convergence we consider the scheme under variance inflation. }
 \item We analyse the TEKI scheme for time-dependent noise and regularization covariances. Assuming that a learning process for the noise and regularization scaling is accessible,  we can still ensure convergence of the scheme as long as the progress of the learning process is fast enough. 
 \item  {We introduce three adaptive Tikhonov algorithms for EKI. The first is based on results by Chung et al. \cite{CE17,SCCKT18} where we adapt the regularization parameter in TEKI based on a bilevel optimization approach.  The second and third adaptive methods take motivation from Bayesian methodologies, namely through the the maximum a-posteriori (MAP) and that of hierarchical EKI \cite{NKC18,CIRS18,FTH01}. This approach is applicable to learn a parametrized covariance matrix as regularization or even the full covariance matrix through its eigen-decomposition.}
 \item Through various numerical experiments, we illustrate that the adaptive regularization methods for TEKI outperform that of both fixed regularization and the vanilla EKI. We test this on numerical examples including a one dimensional linear PDE, and a one dimensional nonlinear PDE of Darcy flow arising in geosciences.
\end{itemize}

The structure of this paper is given as follows. In Section \ref{sec:TEKI} will review and introduce Tikhonov ensemble Kalman inversion, where we extend previous results to the noisy and adaptive case. The presented results are based on the  continuous-time limit of the scheme.  Section \ref{sec:adapt} is devoted to the derivation and implementation of our adaptive Tikhonov procedures. These approaches will then be tested in Section \ref{sec:num} on various numerical examples including both linear and nonlinear models. Finally we conclude our findings in Section \ref{sec:conc}. The proof of our main theorem is presented in the Appendix.

\section{{Stochastic} Tikhonov Ensemble Kalman Inversion}
\label{sec:TEKI}

In this section we introduce Tikhonov regularization for EKI {with noisy perturbations}. We initiate the section with a background, motivated from \cite{CST19}, where we consider deriving a continuous-time limit and analysis based on a noisy case with perturbed observations and fixed regularization parameter. Our analysis will consist of showing the collapse of the ensemble and convergence to the minimizer of the Tikhonov loss functional. In order to do so we will require tools from data assimilation which include variance inflation. Furthermore, we allow time dependence of the assumed noise and regularization covariance, giving the possibilities of applying our convergence results to adaptive schemes for scaling of the ratio between noise and regularization.

\subsection{Ensemble Kalman Inversion}\label{ssec:EKI}
Let $\mathcal H_1 = \R^{d_1}$, $\mathcal H_2 = \R^{d_2}$ be real spaces and consider the possibly nonlinear mapping $F:\mathcal H_1\to\mathcal H_2$. We formulate the EKI algorithm in order to solve an inverse problem of the form
\begin{equation}\label{eq:ip_general}
z = F(w) + \eta,
\end{equation}
where $z\in H_2$ denotes the data and $\eta\sim\mathcal N(0,\Sigma)$ is observational additive Gaussian noise. The algorithm operates by updating an ensemble of particles $\{w^{(j)}_0\}_{j=1}^J$, where $J\ge2$ is the number of ensemble members, through the sample means
\begin{equation*}
\bar w_n = \frac{1}{J}\sum^{J}_{j=1}w^{(j)}_n,\quad \bar{F}_n = \frac{1}{J}\sum^{J}_{j=1}F(w^{(j)}_n), 
\end{equation*}
and sample covariances
\begin{equation*}
\begin{split}
B^{w p}_{n} &= \frac{1}{J}\sum^{J}_{j=1} \bigl(w_n^{(j)} - \bar{w}_n\bigr)\otimes \bigl({F}(w_n^{(j)}) - \bar{{F}}_n\bigr), \\
B^{pp}_{n}  &= \frac{1}{J}\sum^{J}_{j=1}
\bigl({F}(w_n^{(j)}) - \bar{{F}}_n\bigr)
\otimes   \bigl({F}(w^{(j)}_n) - \bar{{F}}_n\bigr),
\end{split}
\end{equation*}
where $\otimes$ denotes the tensor product defined as
\begin{equation*}
z_1\otimes z_2: \mathcal H_1\to\mathcal H_2,\quad \text{with}\ h\mapsto \langle z_2,h\rangle_{\mathcal H_2}\cdot z_1,
\end{equation*}
for $z_1\in\mathcal H_1,\ z_2\in\mathcal H_2$ and euclidean inner product $\langle \cdot, \cdot\rangle_{\mathcal H_2}$ in $\mathcal H_2$.  Further, we define predictions corresponding to our model, by
\begin{equation}
z_{n+1}^{(j)} = F(w_n^{(j)})+\xi_{n+1}^{(j)}, \quad \xi_{n+1}^{(j)} \sim\mathcal N(0,\Sigma).
\end{equation}
and refer to the above procedure as the \textit{prediction ste}. Then we can compute our new ensemble of particles $w^{(j)}_{n+1}$ at iteration $n+1$ through
\begin{equation}
\label{eq:updateU}
w^{(j)}_{n+1} = w^{(j)}_n + B^{w p}_{n} \big(B^{pp}_n +  \Sigma\big)^{-1}\big(z-z^{(j)}_{n+1}\big).
\end{equation}

We denote equation \eqref{eq:updateU} as the discrete form of EKI and consider a continuous-time analogue.
To do so, we first rescale the covariance matrix $\Sigma = h^{-1}\Sigma$ such that it includes a time stepping $h >0$, which in the limit formally leads to
\begin{equation}
\label{eq:sde}
dw^{(j)}_t = B^{w p}(w_t)\Sigma^{-1}\big(z - {F}(w_t^{(j)})\big)\,dt + B^{w p}(w_t)\Sigma^{-1/2}dW_t^{(j)},
\end{equation}
where $W_t^{(j)}$ denote independent Brownian motions in $\mathcal H_2$. Using the definition of $C^{w p}$ and $C^{pp}$, we can also write
\begin{equation}
\label{eq:sde2}
\begin{split}
dw^{(j)} = &\frac1J\sum\limits_{k=1}^J\langle F(w_t^{(k)}-\bar{F}_t,z - {F}(w_t^{(j)})\rangle_{\Gamma}(w_t^{(k)}-\bar{w}_t)\,dt\\
	 + &\frac1J\sum\limits_{k=1}^J\langle F(w_t^{(k)})-\bar{F}_t,\Sigma^{-1/2}dW_t^{(j)}\rangle (w_t^{(k)}-\bar{w}_t).
\end{split}
\end{equation}
For a detailed analysis of the discrete to continuous time limit we refer to \cite{BSWW2021}.  {In order to apply EKI to the original inverse problem \eqref{eq:inv} we set $\mathcal H_1 = \R^{d_u},$ $\mathcal H_2 = \R^K$, $w = u$, $z = y$ and consider the forward map $F\equiv G$.  Assuming that $G (\cdot) = G \cdot$ is linear, the EKI reads as 
\begin{equation}\label{eq:lin_EKI}
du_t^{(j)} =C(u_t)G^\top\Gamma^{-1}\big(y - Gu_t^{(j)}\big)\,dt + C(u_t)G^\top\Gamma^{-1/2}dW_t^{(j)}
\end{equation}
where $C(u)$ denotes the sample covariance in the parameter space $\R^{d_u}$
\[C(u) = \frac{1}{J}\sum^{J}_{j=1} \bigl(u^{(j)} - \bar{u}\bigr)\otimes \bigl(u^{(j)}) - \bar{u}\bigr). \]
 In \cite{} ignoring the diffusion in \eqref{eq:lin_EKI} the EKI has been analysed in a deterministic setting which results in a preconditioned gradient flow of the form
\begin{equation*}
\frac{du_t^{(j)}}{dt} = -C(u_t) \nabla_{u} \Phi(u_t^{(j)};y),
\end{equation*}
where the empirical covariance acts as the preconditioner of the flow seeking a minimization of the objective function $\Phi(u,y) = \frac12\|Gu-y\|^2_\Gamma$.  Nonetheless one direction which has seen limited contributions is analyzing EKI in the noisy case. For this example noisy observations are present and the performance of EKI can considerably deteriorate with the inclusion of noise. One paper aimed at tackling this was \cite{SS17b} where the authors introduced a discrepancy principle of the form
$$
\| {G}(u) - y_\ast \|^2_2 \leq \tau \sqrt{\textrm{Trace}(\Gamma)}, \quad \tau >1,
$$
where $y_\ast$ is the true observed data. The stopping criteria is crucial to ensure no blowup of the system, where well-posedness was able to be proven. Although this is a promising step towards handling the noisy case, there still remains open problems, such as other directions to prevent this and the extension to the stochastic formulation of EKI. Therefore, we are going to consider the incorporation of Tikhonov regularization in EKI and corresponding theoretical analysis in the stochastic formulation.}

\subsection{Incorporation of Tikhonov regularization}

Tikhonov regularization is an effective form of regularization and is well-understood within inverse problems \cite{BB18,EHN96}. In the context of EKI, this form of regularization can be interpreted as Gaussian prior assumption.  The derivation is taken from \cite{CST19} which we recap, where we now refer to this as TEKI. This includes a modification of the inverse problem and the corresponding objective functional of interest. {To do so, we will first introduce the EKI algorithm for general forward problems and reduce the derivation and the theoretical results for TEKI to the results of the original algorithm.}

{In order to incorporate Tikhonov regularization, we extend our original model \eqref{eq:inv} by incorporating prior information extending \eqref{eq:inv} to the equations
\begin{subequations}
\begin{align}
\label{eq:inv_re1}
y&=G(u) + \eta_1, \\ 
\label{eq:inv_re2}
0&= u+\eta_2, 
\end{align}
\end{subequations}
where $\eta_1, \eta_2$ are independent random variables
distributed as $\eta_1 \sim \mathcal{N}(0,\Gamma)$ and $ \eta_2 \sim \mathcal{N}(0, {{\lambda}}^{-1}C_0)$, where $\Gamma\in\R^{K\times K}$, $C_0\in\R^{d_u\times d_u}$ are positive definite and $\lambda>0$ denotes the regularization parameter. Let $\mathcal H_1 = \R^{d_u}$ and $\mathcal H_2=\R^{K} \times \R^{d_u}$, then we define the variables $z,\eta$ 
and mapping $F: \R^{d_u} \mapsto \R^{K}\times \R^{d_u} $ as follows:
\[
z=\begin{bmatrix}
y\\
0
\end{bmatrix},\quad
F(u)=\begin{bmatrix}
{G}(u)\\
u
\end{bmatrix},
\quad
\eta=\begin{bmatrix}
\eta_1\\
\eta_2
\end{bmatrix},
\]
noting that then
\[
\eta \sim N(0,\Sigma), \quad
\Sigma =
\begin{bmatrix}
\Gamma & 0\\
0 & {{\lambda}}^{-1}C_0
\end{bmatrix}.
\]
Consequently we now consider the EKI algorithm for \eqref{eq:ip_general} where $u$ is now playing the role of $w$ and obtain with \eqref{eq:updateU} the TEKI update formula}
\begin{equation}
\label{eq:updateT}
u^{(j)}_{n+1} = u^{(j)}_n + B^{up}_{n} \big(B^{pp}_n +  \Sigma\big)^{-1}\big(z-z^{(j)}_{n+1})\big),
\end{equation}
with i.i.d.~predictions
\begin{equation}
\label{eq:dataT}
z^{(j)}_{n+1} = F(u^{(j)})+ \zeta^{(j)}_{n+1}, \quad \zeta^{(j)}_{n+1} \sim 
\mathcal{N}(0,\Sigma),
\end{equation}
where we include pertubation in the observation space as well as in the parameter space.
The associated loss function is given by 
\begin{equation}
\label{eq:loss3}
\ell_Z(z',z)=\frac12 \|z'-z\|_{\Sigma}^2.
\end{equation}
Alternatively we can express \eqref{eq:loss3} as a loss function of the original unknown $u$ with regularization as
\begin{equation}
\calI(u;y) := \frac{1}{2}\|y - G(u)\|_{\Gamma}^2 + \frac{{{\lambda}}}{2}  \|u\|_{C_0}^2.
\end{equation}
We now proceed with the analysis of the continuous-time limit of the TEKI
\begin{equation}
\label{eq:cts_com}
du^{(j)}_t = B^{up}(u_t) \Sigma^{-1}(z-F(u_t^{(j)}))dt + B^{up}(u_t)(\Sigma)^{-1/2}dW^{(j)}_t,
\end{equation}
where $W_t^{(j)}$ denotes independent Brownian motions on the space $\R^{K}\times \R^{d_u}$. We will denote the filtration introduced by the particle dynamics as $\mathcal{F}_t=\sigma(u_s,s\le t)$. 

Suppressing the dependance on $t$, we obtain, with the definition of the sample covariance, 
\begin{equation}\label{eq:TEKI_SDE}
\begin{split}
du^{(j)}= &\frac{1}{J}\sum^{J}_{k=1}(u^{(k)} - \bar{u})\langle F(u^{(j)}) - \bar{F}, \Sigma^{-1}(z-F(u^{(j)}))\rangle\,dt \\
			&+\frac{1}{J}\sum^{J}_{k=1}(u^{(k)} - \bar{u})\langle F(u^{(j)}) - \bar{F}, \Sigma^{-1/2} dW^{(j)}\rangle.
\end{split}
\end{equation}
and using the definition of $F$ we can similarly write
\begin{align*}
du^{(j)} = &\frac{1}{J}\sum^{J}_{k=1} \bigg(\langle G(u^{(j)})-\bar{G},\Gamma^{-1}(y-G(u^{(j)}))\rangle - \langle u^{(k)}-\bar{u},\lambda C_0^{-1} u^{(j)}\rangle \bigg)(u^{(k)} - \bar{u})\\
			&+\frac{1}{J}\sum^{J}_{k=1} \bigg(  \langle G(u^{(j)})-\bar{G},\Gamma^{-1/2} d{\hat W}^{(j)}\rangle \bigg)(u^{(k)} - \bar{u})\\
			&+ \frac{1}{J}\sum^{J}_{k=1} \bigg(  \langle u^{(j)})-\bar{u},\sqrt{\lambda}C_0^{-1/2} d{\tilde W}^{(j)}\rangle \bigg)(u^{(k)} - \bar{u}),
\end{align*}
where $\hat W^{(j)}$ and $\tilde W^{(j)}$ denote the first $K$ and last $d_u$ components of $W$ respectively.  {By definition} they are
independent Brownian motions.
\begin{remark}
{We note that based on the continuous-time limit the subspace property for the ensemble of particles \cite{ILS13} can be verified. Denoting by $S\subset \R^{d_u}$ the linear span of $\{u_0^{(j)}\}_{j=1}^J$ the particle systems remains in $S$, i.e.~$u_t^{(j)}\in S$ for all $t\ge0$ and $j=1,\dots,J$. Due to the subspace property we can transfer our presented results to a coordinate system in $S$ giving the possibility to generalize the presented theory to a general Hilbert space setting for the underlying parameter space $X$.}
\end{remark}

\subsection{Linear setting: Continuous-time analysis for fixed regularization}
{For our theoretical analysis we start with a fixed choice of regularization with $\lambda = 1$ and fixed regularization matrix $C_0$.  Note that we can always rescale $C_0$ in order to cover alternative choices of $\lambda>0$. In contrast, later we are going to analyse TEKI with a time-depending choice $C_t$ which allows to incorporate adaptive learning schemes of the regularization parameter $\lambda \mapsto \lambda_t$ as well as learning schemes for the whole regularization matrix $C_0\mapsto C_t$. Additionally, we emphasize that we also allow for time depending noise scaling through a time depending choice of the noise covariance matrix $\Gamma_t$. This can be generalized to a time depending $\Sigma_t$. }
To obtain further insight, we assume that $G(\cdot) = A\cdot$ is a linear forward operator with $A\in\R^{d_u\times K}$ leading to the following SDE
\begin{equation}
\label{eq:limit}
\begin{split}
du^{(j)}&= C(u) A^\top \Gamma^{-1}(y-Au^{(j)})dt - C(u)C_0^{-1}u^{(j)}dt\\ &\quad + C(u) A^\top \Gamma^{-1/2} d{\hat W}_t^{(j)} + C(u) C_0^{-1/2}d\tilde W_t^{(j)}.
\end{split}
\end{equation}
The following proposition establishes the existence and uniqueness of strong solutions for \eqref{eq:limit}.
\begin{proposition}
\label{prop:existence}
Consider an initial ensemble $u=\{u_0^{(j)}\}_{j=1}^J$ of $\cF_0$-measurable maps $u_0^{(j)}:\Omega\to X$ which are almost surely linearly independent. Then for the set of coupled SDEs given in \eqref{eq:limit} there exists for all $T\ge0$ a unique strong solution $(u_t)_{t\in[0,T]}$.
\end{proposition}
\begin{proof}
Since the set of coupled SDEs \eqref{eq:limit} can be viewed by \eqref{eq:TEKI_SDE} with $F:\R^{d_u} \to\R^K\times\R^{d_u} $ with
\begin{equation*}
F(u)=\begin{bmatrix}
A\\
I
\end{bmatrix}u,
\end{equation*}
which is again a bounded and linear map, the result follows by the application of Theorem 3.4 in \cite{BSWW19}. 
\end{proof}

\subsubsection{Quantification of the ensemble collapse}
\label{ssec:spread}
In the following, we will denote the spread of the ensemble of particles by $e^{(j)} = u^{(j)}-\bar u$ which will converge to zero as $t\to\infty$ in $L^2$ with a given rate.  This means, the ensemble collapses in time to its mean in the parameter space. 
\begin{proposition}
Let $u_0=\{u_0^{(j)}\}_{j=1}^J$ be $\cF_0$-measurable maps $u_0^{(j)}:\Omega\to X$ such that $C_0 = \mathbb{E}\bigg[\frac1J\sum\limits_{j=1}^J|\Sigma^{-\frac12}F(e_0^{(j)})|^2\bigg]<\infty$, then it holds true that
\begin{equation*}
\mathbb{E}\bigg[\frac1J\sum\limits_{j=1}^J|\Sigma^{-\frac12}F(e_t^{(j)})|^2\bigg]\le \frac{1}{C_0^{-1}+\frac{J+1}{J^2}t}.
\end{equation*}
Furthermore, it follows
\begin{equation*}
\mathbb{E}\bigg[\frac1J\sum\limits_{j=1}^J|e_t^{(j)}|^2\bigg]=\mathcal{O}\left(t^{-1}\right),
\end{equation*}
where $h(t)=\mathcal O(t^{-1})$ means that $h(t)$ decreases asymptotically of order $t^{-1}$ following the big O notation.
\end{proposition}

\begin{proof}
The first assertion follows by Theorem 4.2 in \cite{BSWW19} and the second assertion follows by the definition of $F$.
\end{proof}
\subsubsection{Convergence of the regularized loss function}
The TEKI can be applied to minimize the regularized loss function given by 
\begin{equation*}
\mathcal{I}(u,y) = \frac12\|\Sigma^{-\frac12}(z-F(u))\|_{\R^{K}\times \R^{d_u} }^2 = \frac12 \|y-Au\|_{\Gamma}^2+\frac{1}{2}\|u\|_{C_0}^2.
\end{equation*}
We consider the residuals defined as
\begin{equation*}
\widetilde{r}^{(j)}_t := \Sigma^{-1/2}F(u_t^{(j)}-u^\ast) = \begin{pmatrix}
\Gamma^{-1/2}A(u_t^{(j)}-u^\ast)\\ C_0^{-1/2} (u_t^{(j)}-u^\ast)
\end{pmatrix},
\end{equation*}
where $u^\ast$ is the global minimizer of $\mathcal{I}$, i.e. $u^\ast$ is given by
\begin{equation*}
u^\ast:= \left(A^\top\Gamma^{-1}A+C_0^{-1}\right)^{-1}A^\top\Gamma^{-1}y,
\end{equation*}
and satisfies
\begin{equation*}
\nabla \mathcal{I}(u^\ast,y) = 0.
\end{equation*}
Hence, using the gradient structure of \eqref{eq:limit} and the fact that $\nabla \mathcal{I}(u^\ast,y) = 0$, we can write the dynamics of $\tilde r_t^{(j)}$ by
\begin{align*}
d\tilde r_t^{(j)} &= d\left( \Sigma^{-1/2}F(u_t^{(j)}-u^\ast)\right)\\ & = \Sigma^{-1/2}F B(u_t) F^\top \Sigma^{-1/2} \Sigma^{-1/2}(F(u_\ast-u_t^{(j)})\,dt + \Sigma^{-1/2}F B(u_t) F^\top \Sigma^{-1/2}\,dW_t^{(j)}\\
 &= -C(\tilde r_t^{(j)})\tilde r_t^{(j)}dt + C(\tilde r_t^{(j)})\,dW_t^{(j)}.
\end{align*}
Based on these derivations, we can apply Proposition~5.1 in \cite{BSWW19} in order to prove monotonicity of the quantity
\begin{equation}
\frac1J\sum_{j=1}^J \E[\|\tilde r_t^{(j)}\|^2] = \frac1J\sum_{j=1}^J \E[\|\Gamma^{-1/2} A(u_t^{(j)} - u_\ast)\|^2] + \E[\|C_0^{-1/2} (u_t^{(j)}-u_\ast)\|^2],
\end{equation}
i.e.~we have the following result:

\begin{proposition}\label{thm:residuals}
Assume that $y$ are noisy measurements of the true parameter $u^\dagger$ under $A$, i.e. $y = Au^\dagger+\eta^\dagger,$ where $\eta^\dagger\in\R^K$ denotes a realization of noise and let $u_0 = (u^{(j)}_0)$ be $\mathcal{F}_0$-measurable maps $u_0^{(j)}:\Omega \rightarrow \R^K$ such that we have bounded moments $\mathbb{E}\bigg[\frac1J\sum\limits_{j=1}^J |\widetilde r_0^{(j)}|^2\bigg] < \infty$. Then  $\mathbb{E}\left[\frac1J\sum\limits_{j=1}^J |\widetilde r_t^{(j)}|^2\right]$ is strictly monotonically decreasing in time.
\end{proposition}
While the application of Proposition~5.1 in \cite{BSWW19} to the original EKI algorithm does only provide monotonic decrease of the residuals for noise-free data, the incorporation of Tikhonov regularization leads to stability w.r.t.~noisy data.

\subsubsection{Variance inflation} 
To ensure convergence of the Tikhonov regularized loss function we will incorporate variance inflation into the algorithm, which can be shown to aid with stability for both EKI and data assimilation \cite{JLA07,JLA09,SS17}. The resulting scheme provides converges to the minimizer of $\mathcal{I}(u,y)$.
Under the assumption of the forward problem being linear and $C_0$ being strictly positive definite, it follows that the loss function $\mathcal{I}$  is strongly convex, since 
\begin{equation*}
\nabla^2 \mathcal I = A^\top\Gamma^{-1}A+ C_0^{-1}>0.
\end{equation*}
We will denote the smallest eigenvalue of $\nabla^2\mathcal{I}$ by $\kappa_{\min}>0$.  We will incorporate a variance inflation reduced over time into the system of SDEs \eqref{eq:limit} in the following way
\begin{equation}
\label{eq:limit_VI}
\begin{split}
du_t^{(j)} = &\left(C(u_t)+\frac{1}{t^\alpha+c}B\right)\left(A^\top\Gamma^{-1}(y-Au_t^{(j)})+C_0^{-1}u_t^{(j)}\right)\,dt\\
&+C(u)A^\top\Gamma^{-1/2}\,d{\hat W}_t^{(j)} + C(u)C_0^{-1/2}\, d\tilde W_t^{(j)},
\end{split}
\end{equation}
where $\alpha\in(0,1)$, $c>0$ and $B$ denotes a strictly positive definite matrix.  Similarly as before, we can now write the dynamics for $\tilde r_t^{(j)}$ by 
\begin{equation*}
d\tilde r_t^{(j)} = -(C(\tilde r_t^{(j)})+\frac{1}{t^\alpha+c} \tilde B)\tilde r_t^{(j)} + C(\tilde r_t^{(j)})\,dW_t^{(j)},
\end{equation*}
with $\tilde B = \Sigma^{-1/2} F BF^\top \Sigma^{-1/2}$ positive definite.  The application of Theorem~5.2. in \cite{BSWW19} leads to the following convergence result:
\begin{equation*}
\mathbb{E}\bigg[\frac1J\sum\limits_{j=1}^J |\widetilde r_t^{(j)}|^2\bigg]=\mathbb{E}\bigg[\frac1J\sum\limits_{j=1}^J |\Gamma^{-1/2} A(u_t^{(j)}-u^\ast)|^2 + |C_0^{-1/2}(u_t^{(j)}-u^\ast)|^2\bigg]\in \mathcal O(t^{-(1-\alpha)}).
\end{equation*}

\subsection{Extension to time-varying $\Sigma_t$}

The above presented results are specific to the case of fixed regularization parameters $\lambda$ and $C_0$ as well as fixed noise covariance $\Gamma$. As we are interested in choosing the regularization iteratively, we now seek to extend the convergence results for TEKI to the case of time-varying $\Sigma_t$. This will leads to an SDE for $u_t^{(j)}$ similar to \eqref{eq:limit_VI}, expect that $C_0$ is replaced by a time depending $C_t$.

The advantage of the following analysis is to transfer convergence results of TEKI with fixed choices of regularization and noise to adaptive changes of these quantities in time.  In particular, assuming to have an adaptive choice $C_t$ for the regularization matrix, including the choice of a regularization parameter $\lambda_t>0$ such that $C_t = \frac{1}{\lambda_t} C_0$ as well as an adaptive choice of the noise covariance $\Gamma_t$, we are interested in minimizing
\begin{equation*}
\mathcal I_t(u) = \frac12\|\Sigma_t^{-1/2}(z-F(u))\|^2 = \frac12\|y-Au\|_{\Gamma_t}^2 + \frac12\|u\|_{C_t}^2,
\end{equation*}
as time approaches infinity.  For simplicity, we assume that the learning progress of the adaptive choices $\Gamma_t$ and $C_t$ is deterministic and fast enough. To be more precise, we take the following assumption.
\begin{assumption}
\label{aspt:convGamma}
Suppose the following holds a.s.:  The learning process $\Psi_t=\frac{d}{dt}\Sigma_t^{-1}$ satisfies for some $\beta>1$ and $\kappa_1,\ \sigma_{\min}>0$ that
\begin{equation}
\label{eqn:asptconvGamma}
 \Sigma_t^{-1}\succeq \sigma_{\min} I,\quad-{\frac{1}{t^{\beta}+R}}\Sigma_t^{-1}\preceq \Psi_t\preceq {\frac{1}{t^{\beta}+R}}\Sigma_t^{-1} \quad \text{and} {\quad \int_0^t \|\Psi_t\|\, dt\le \kappa_1}.
\end{equation}
Furthermore, we assume that the smallest eigenvalues of $\tilde B_t = \Sigma_t^{-1/2} F B F^\top \Sigma_t^{-1/2}$ are bounded from below uniformly in time by $\sigma_{\min}$.
\end{assumption}
{The first requirement of \eqref{eqn:asptconvGamma} is that $\Sigma_t$ should not not be singular. 
This is an intuitive requirement, since otherwise $\mathcal{I}_t(u)$ can be infinite. The requirement for $\tilde B_t$ is imposed because of similar reasons. The second and third requirements in 
\eqref{eqn:asptconvGamma} are imposed on $\Psi_t=\frac{d}{dt}\Sigma_t^{-1}$. Essentially, they require the learning process cannot change $\Sigma_t$ too abruptly. In practice, this can usually be achieved by using an appropriate step size in the optimization algorithms.  }

Under this assumption we can write the global minimum of
\begin{equation*}
\mathcal I_t(u) = \frac12\|\Sigma_t^{-1/2}(z-F(u))\|^2,
\end{equation*}
depending on $\Sigma_t$ as
\begin{equation*}
u_t^\ast := u^\ast(\Sigma_t) = \left(A^\top\Gamma_t^{-1}A+C_t^{-1}\right)^{-1}A^\top\Gamma_t^{-1}y.
\end{equation*}
Similar as $\Sigma_t$ itself, we can also write the evolution equation in terms of noise covariance $\Gamma_t$ and regularization covariance $C_t$ as two separated evolution equations
\begin{equation*}
\Psi_t^\Gamma =  \frac{d}{dt}\Gamma_t,\quad \Psi_t^C = \frac{d}{dt} C_t.
\end{equation*}
In the following, we are interested in the dynamical behavior of
\begin{equation*}
\tilde r_t^{(j)} =\Sigma_t^{-1/2}\mathfrak r_t^{(j)} =\Sigma_t^{-1/2} F(u_t^{(j)} - u_t^\ast),
\end{equation*}
where $u_t^{(j)}$ is a strong solution of
\begin{equation*}
\begin{split}
du_t^{(j)} &= \left(C(u_t)+\frac{1}{t^\alpha+c}B\right)\left(A^\top\Gamma_t^{-1}(y-Au_t^{(j)})+C_t^{-1}u^{(j)}\right)\,dt\\
&\quad +C(u_t)A^\top\Gamma_t^{-1/2}\,d{\hat W}_t^{(j)} + {C(u)C_t^{-1/2}\, d\tilde W_t^{(j)}}\\
			&= \left(C(u_t)+\frac{1}{t^\alpha+c}B\right)F^\top \Sigma_t^{-1}(z-Fu_t^{(j)})dt + C(u_t)F^\top\Sigma_t^{-1/2}dW^{(j)}_t.
\end{split}
\end{equation*}

\begin{thm}\label{thm:residuals_VI}
Suppose Assumption \ref{aspt:convGamma} is satisfied and
assume that $y$ are noisy measurements of the true parameter $u^\dagger$ under $A$, i.e. $y = Au^\dagger+\eta^\dagger,$ where $\eta^\dagger\in\R^K$ denotes a realization of noise. Furthermore, let $u_0 = (u^{(j)}_0)$ be $\mathcal{F}_0$-measurable maps $u_0^{(j)}:\Omega \rightarrow \R^K$ such that we have bounded moments $\mathbb{E}\bigg[\frac1J\sum\limits_{j=1}^J |\widetilde r_0^{(j)}|^2\bigg] < \infty$ and let $B \in \mathcal{L}(\R^{d_u} ,\R^{d_u} )$ be a strictly positive definite operator. Then for all $\alpha>0$  it holds true that
\begin{equation*}
\mathbb{E}\bigg[\min_{s\leq t}\frac1J\sum\limits_{j=1}^J |\widetilde r_s^{(j)}|^2\bigg] \in\mathcal{O}\left(t^{-(1-\alpha)}\right).
\end{equation*}
\end{thm}
\begin{proof}
The proof is deferred to the Appendix.
\end{proof}

\section{Adapting the regularization parameter}
\label{sec:adapt}

An important point to consider, in the theory of regularization for inverse problems, is the choice of the regularization parameter. The parameter itself can depend largely on the problem itself and the specific form of regularization \cite{BB18,EHN96}. In this section, we describe various ways to find a good choice of the Tikhonov parameter ${\lambda}$. {In Figure \ref{fig:adaptive_reg} we describe the task of adapting the regularization parameter within the algorithm of EKI. 
\begin{figure}[h!]
\includegraphics[width=0.90\textwidth]{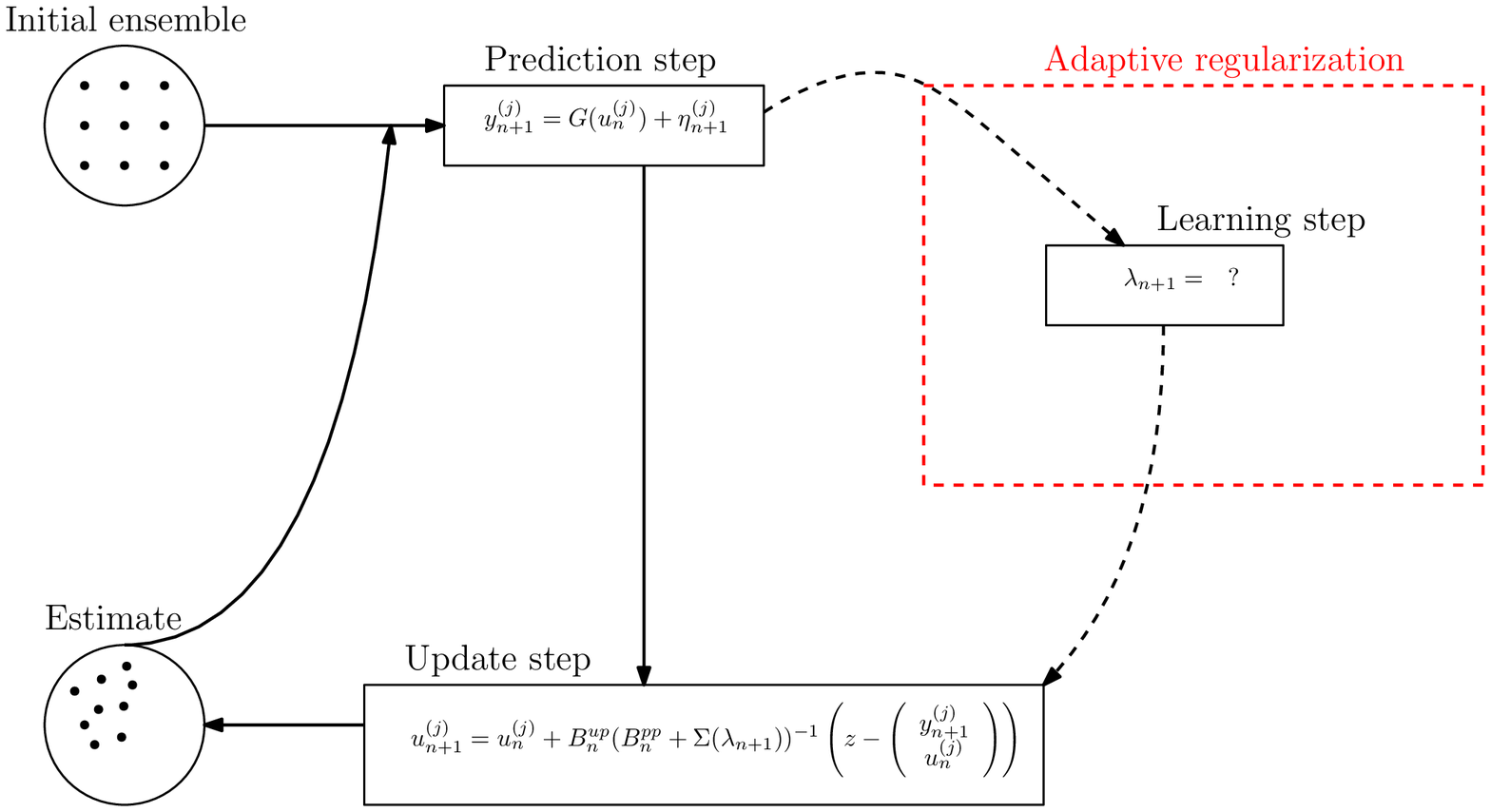}
\caption{Representation of adaptive regularized ensemble Kalman inversion.}
\label{fig:adaptive_reg}
\end{figure} 
While in the algorithm presented in Section \ref{sec:TEKI} we kept the regularization parameter fixed, we now consider different approaches where we adapt the regularization parameter in-between the prediction step and the update step.  {As a result, the assumed regularization covariance matrix varies in time $C_t = \frac{1}{\lambda_t}C_0$. We further note, that including a noise scaling parameter $\Gamma_t \to \frac{1}{\gamma_t} \Gamma$ our proposed methods give the possibility to adapt the ratio between noise and regularization by considering $C_t= \frac{\gamma_t}{\lambda_t} C_0$ and fixed noise covariance $\Gamma$.} In the following, we motivate three different adaptive procedures:
\begin{itemize}
\item The first method is based on a bilevel optimization problem. Here we will use our prediction step to generate artificial training data which will be used to adapt the regularization parameter minimizing the distance to the corresponding Tikhonov solutions of the training data. The application of this method can be interpreted as a parametric bootstrapping approach.
\item The second method is based on the MAP formulation in the Bayesian framework of inverse problems, where the regularization parameter $\lambda$ is treated as a scaling of a Gaussian prior covariance matrix.
\item The third method is motivated through ideas of hierarchical Bayesian methods, where we treat the regularization parameter as hyperparameter of the underlying Gaussian prior covariance matrix. This approach additionally extends to multiple regularization parameter and it is even possible to update the whole covariance matrix via a diagonalization.
\end{itemize}
}

\subsection{Bilevel learning applied to EKI: Algorithmic approach} \label{ssec:bilevel_linear} 
One difficulty in the search of optimal tuning parameter $\lambda$ for EKI is how to quantify the ``\textit{goodness}" of a given $\lambda$. 
One natural criterion is using the generalization error, which can be defined as 
\[
l(\lambda)=\E \mathcal{L}(u_{\lambda}(Y),U). 
\]
where $Y$ is a random sample generated by the observation model $Y=G(U)+\eta$. 
In our previous work \cite{CSTW20}, we have assumed the existence of i.i.d. training samples from the joint distribution of $(U,Y)$,  such that $l(y)$ can be approximated by the Monte Carlo average. In this paper, we do not assume the existence of such training samples and seek a different approach known as parametric bootstrap. In statistics, one first builds a parametric model $Z\sim p_{\theta,\lambda}$, and then find the optimal parameter $\hat{\theta}$ and tuning parameter $\lambda$ from data $z^1,\ldots,z^n$. It is often of interest to estimate error $\E |f(\theta)-f(\hat{\theta})|^2$ for some test function $f$, but this quantity is often not accessible if $n$ is too small. The parametric bootstrap method considers generating different batches of artificial data $D^{(i)}=\{\hat{z}^{i,1},\ldots, \hat{z}^{i,n}\}$ from the distribution $p_{\hat{\theta},\lambda}$, and use each batch to produce an estimation $\hat{\theta}^{(i)}$. Then an estimator of the mean square error (MSE) can be obtained by 
\[
\frac1 M\sum_{i=1}^M\E |f(\hat{\theta}^{(i)})-f(\hat{\theta})|^2.
\]
In EKI, $u$ represents the $\theta$ we want to estimate, and we have only $n=1$ real data $y$, therefore we see $u^j$ as a learnt result. To evaluate how good $u^{(j)}$ and $\lambda$ are, we apply bootstrap and generate a data $y^{(j)}=G(u^{(j)})+\eta^{(j)}$, and then find $\hat{u}^{(j)}$. The error is given by $\mathcal{L}(u^{(j)},\hat{u}^{(j)})$.

Recall that through the Bayesian setting of the inverse problem
\begin{equation*}
y = G(u)+\eta,
\end{equation*}
we view $u$ and $\eta$ as independent random variables distributed by $\mathcal{N}(0,\lambda^{-1}C_0)\otimes \mathcal{N}(0,\Gamma)$. To get access to training data, we can draw $(u^{(j)})_{j=1}^J$ samples of the prior distribution, $(\eta^{(j)})_{j=1}^J$ realizations of the noise and compute
\begin{equation*}
y^{(j)} = G(u^{(j)})+\eta^{(j)}.
\end{equation*}

To incorporate those ideas of learning the regularization parameter from \textit{training data}, we will give an alternative view point of EKI with perturbed observations. Instead of considering perturbations directly to the true observation as in \eqref{eq:dataT}, we will now view the perturbation as producing training data in each iteration. Furthermore, instead of computing an optimal regularization parameter only at the beginning of the methods, we assume in each iteration that our current ensemble of particles represents current prior information in the form of an empirical distribution
\begin{equation*}
\mu_n = \frac{1}{J}\sum\limits_{j=1}^J \delta_{u_n^{(j)}}.
\end{equation*}

\subsubsection{Linear setting}
\label{subsec:adaptlinear}
We view $(u,\eta)\sim \mu_n\otimes \mathcal{N}(0,\Gamma)$ and compute
\begin{equation*}
\widehat{\lambda}_{n+1}^J = \underset{\lambda}\argmin\ \E_{(u,\eta)}[\|(A^\top\Gamma^{-1}A+\lambda C_0^{-1})^{-1}A^\top\Gamma^{-1}(Au+\eta)-u\|^2],
\end{equation*}
where we approximate
\begin{align*}
\E_{(u,\eta)}[\|(A^\top\Gamma^{-1}A+\lambda C_0^{-1})^{-1}A^\top\Gamma^{-1}(Au+\eta)-u\|^2] \\
\approx \frac1J \sum\limits_{j=1}^J\|(A^\top\Gamma^{-1}A+\lambda C_0^{-1})^{-1}A^\top\Gamma^{-1}y_{n+1}^{(j)}-u_n^{(j)}\|^2.
\end{align*}

Our training data is produced by perturbing the particles mapped by the forward operator,
\begin{equation}
\label{eq:train}
y^{(j)}_{n+1} = G(u_n^{(j)}) + \eta^{(j)}_{n+1}.
\end{equation}

Following the ideas of \cite{AMOS16,CE17} we now employ a way to update the regularization parameter ${\lambda}_n$ in each iteration where we will do gradient descent in each update step w.r.t.~the loss function
$$ \frac1J \sum\limits_{j=1}^J\|(A^\top\Gamma^{-1}A+\lambda_n C_0^{-1})^{-1}A^\top\Gamma^{-1}y_{n+1}^{(j)}-u_n^{(j)}\|^2,$$
depending on ${\lambda}_n$. To do so we will make use of an error defined as
\begin{equation}
\label{eq:err_train}
v^{(j)}_{n+1}({\lambda}) := T_{{\lambda}}(y_{n+1}^{(j)}) - u^{(j)}_n,
\end{equation}
which represents the difference of the current particle $u_n^{(j)}$ to the minimizer of the Tikhonov regularized loss function. With this we use our particle system to construct training data which we can use to learn the regularization parameter $\lambda$ adaptively.

From \eqref{eq:err_train} we can define the loss function
\begin{equation}
\label{eq:f}
f_{n+1}({\lambda}) := \frac{1}{J} \sum^{J}_{j=1} \frac12\| v^{(j)}_{n+1}({\lambda}) \|^2,
\end{equation}
where for simplicity we drop the dependence of $n$ and $j$. To {do a gradient descent} step w.r.t.~\eqref{eq:f} we need to compute its derivative which means we also need to compute the derivative of \eqref{eq:err_train}. This will be important for the implementation when we construct our numerical examples. 
To proceed we compute both $f'({\lambda})$ and $v'({\lambda}) $. To aid we use the following formula,
\begin{equation}
\label{eq:relation}
\frac{d \|v_{n+1}^{(j)}({\lambda})\|^2}{d {\lambda}} = (v^{(j)}_{n+1})^\top({\lambda}) \cdot  (v^{(j)}_{n+1})'({\lambda}),
\end{equation}
and
\begin{align*}
{(v_{n+1}^{(j)})'({\lambda}) = \frac{d u_{n+1}^{(j)}({\lambda})}{d {\lambda}}}&{= \frac{d(A^\top\Gamma^{-1}A+{\lambda} C_0^{-1})^{-1}A^\top\Gamma^{-1}y^{(j)}_{n+1}}{d{\lambda}},} \\
&=  \frac{d(A^\top\Gamma^{-1}A+{\lambda} C_0^{-1})^{-1}}{d{\lambda}} A^\top\Gamma^{-1}y^{(j)}_{n+1},\\
&{= -(A^\top\Gamma^{-1}A+{\lambda} C_0^{-1})^{-1}C_0^{-1}(A^\top\Gamma^{-1}A+\lambda C_0^{-1})^{-1}} \\ &{\cdot A^\top \Gamma^{-1}y^{(j)}_{n+1}.}
\end{align*}
Therefore, using the expression for the derivative of $v$, we can now express the derivative of \eqref{eq:relation} as
\begin{align*}
\frac{d \|v_{n+1}^{(j)}({\lambda})\|^2}{d {\lambda}} = &-\left((A^\top \Gamma^{-1}A + {\lambda} C_0^{-1})^{-1}A^\top \Gamma^{-1} y^{(j)}_{n+1} - u^{(j)}_n\right) (A^\top\Gamma^{-1}A+{\lambda} C_0^{-1})^{-1} \\ &\cdot C_0^{-1}
 (A^\top\Gamma^{-1}A+{\lambda} C_0^{-1})^{-1} A^\top \Gamma^{-1}y^{(j)}_{n+1}.
\end{align*}

\subsubsection{Nonlinear setting}
\label{subsec:adaptnonlinear}

While we have used the closed expression of the minimizer of the Tikhonov functional, we are not able to use this expression in the nonlinear setting. To avoid this issue, we will present another way of choosing the regularization parameter adaptively.

For our first method, we will make use of the data-driven regularization approach \cite{AMOS16,CSTW20}. In particular, we consider the bilevel optimization problem in a general nonlinear setting with Tikhonov regularization, i.e.
\begin{equation}
\begin{split}
\widehat\lambda &\in \underset{\lambda>0}{\argmin} \ \mathbb{E}_{\mu(U,Y)}[|R_\lambda(Y)-U|^2],\\
R_\lambda(Y)&:= \underset{u\in\R^{d_u}}{\argmin} \ \frac12\|G(u)-Y\|_\Gamma^2+\frac12\|u\|_{C_0}^2,
\end{split}
\end{equation}
We assume that we have given the current ensemble of particles $(u_n^{(j)})_{j=1}^J$, which  represent current information about the unknown true parameter $u^\dagger$. Furthermore, assume that we have given a current regularization parameter ${\lambda}_n$. The method is similarly to the previous one based on learning the regularization parameter over time with the help of artificial training data $(u_n^{(j)},y_{n+1}^{(j)})$, constructed in the prediction step. The update step \eqref{eq:updateT} pushes the current ensemble to $(u_{n+1}^{(j)})_{j=1}^J$ in order to get closer to the minimizer of the Tikhonov functional, i.e. into direction of $R_{\lambda_{n+1}}(y_{n+1}^{(j)})$. In the linear setting we have chosen ${\lambda}_{n+1}$ minimizing the difference. Using an empirical approximation we aim to choose $\lambda_{n+1}$ minimizing the difference
\[ v_{n+1}(\lambda)= \frac1J \sum\limits_{j=1}^J \frac12 |u_n^{(j)}-R_{{\lambda}}(y_{n+1}^{(j)})|^2.\]
Since the forward model is assumed to be nonlinear, in general we are not able to compute the Tikhonov solution $R_{{\lambda}}(y_{n+1}^{(j)})$ in closed form. To overcome this issue, we propose to introduce a linearization around the mean of the particle system in each iteration. In particular, we approximate the forward model $G(\cdot)$ by
$$G(u) = G(\bar u_n) + {\rm D}G(\bar u_n) (u - \bar u_n),$$
where ${\rm D}G$ denotes the derivative of $G$ w.r.t.~$u$. Hence, defining $A_{\rm{apprx}} = {\rm D}G(\bar u_n)$ and $a_{\rm{apprx}} = G(\bar u_n) -{\rm D}G(\bar u_n) \bar u_n$, we approximate $R_\lambda(Y)$ by
\begin{align*}
\widehat R_\lambda(Y) &=  \underset{u\in\R^{d_u} }{\argmin} \ \frac12\|A_{\rm{apprx}}-(Y-a_{\rm{apprx}})\|_\Gamma^2+\frac12\|u\|_{C_0}^2\\ &= (A_{\rm{apprx}}^\top \Gamma^{-1}A_{\rm{apprx}}+\lambda C_0^{-1})^{-1}A_{\rm{apprx}}^\top \Gamma^{-1}(Y-a_{\rm{apprx}}).
\end{align*}
Given the traing data produced by perturbing the particles mapped by the forward model 
$$y_{n+1}^{(j)} = G(u_n^{(j)})+\eta_{n+1}^{(j)},$$
we apply compute $\tilde y_{n+1}^{(j)} = y_{n+1}^{(j)} - a_{\rm{apprx}}$ and apply the previously introduced approach in Section~\ref{ssec:bilevel_linear} for linear forward models with $A=A_{\rm{apprx}}$, i.e.
$$\lambda_{n+1} = \lambda_n - \gamma_n \cdot f'(\lambda_n),$$
where 
\begin{align*}
f'(\lambda_n)= &-\left((A_{\rm{apprx}}^\top \Gamma^{-1}A_{\rm{apprx}} + {\lambda} C_0^{-1})^{-1}A_{\rm{apprx}}^\top \Gamma^{-1} \tilde y^{(j)}_{n+1} - u^{(j)}_n\right)\\ &\cdot (A_{\rm{apprx}}^\top\Gamma^{-1}A_{\rm{apprx}}+{\lambda} C_0^{-1})^{-1} \\ &\cdot C_0^{-1}
 (A_{\rm{apprx}}^\top\Gamma^{-1}A_{\rm{apprx}}+{\lambda} C_0^{-1})^{-1} A_{\rm{apprx}}^\top \Gamma^{-1}\tilde y^{(j)}_{n+1}.
\end{align*}

\begin{algorithm}[H]
\caption{Tikhonov EKI: nonlinear adaptive learning regularization}
\label{alg:TEKIadapt_nonlinear}
\begin{algorithmic}[1]
\State \textbf{Input} $\{u^{(j)}_0\}_{j=1}^{J} \sim \mathcal{N}(0,C_0)$,  $J\ge2$, $\lambda_0>0$

For{$n=0,\ldots,N-1$},
\State linearize 
$$A_{\rm{apprx}} = {\rm D}G(\bar u_n),\quad a_{\rm{apprx}} = G(\bar u_n) -{\rm D}G(\bar u_n) \bar u_n.$$
\State construct training data $$\widetilde y_{n+1}^{(j)} = G(u_{n}^{(j)})+\eta_{n+1}^{(j)}-a_{\rm{apprx}}.$$
\State compute
$$
\lambda_{n+1} = \lambda_n - \gamma_n \cdot f'(\lambda_n).
$$
\State update the ensemble of particle by TEKI
$$
u^{(j)}_{n+1} = u^{(j)}_n + B^{up}_{n} \big(B^{pp}_n + \Sigma({\lambda}_{n+1})\big)^{-1}\bigg(z - \begin{pmatrix} y_{n+1}^{(j)} \\ u_n^{(j)}\end{pmatrix}\bigg).
$$
\State \textbf{end}
\end{algorithmic}
\end{algorithm}

\begin{figure}[h!]
\includegraphics[width=0.90\textwidth]{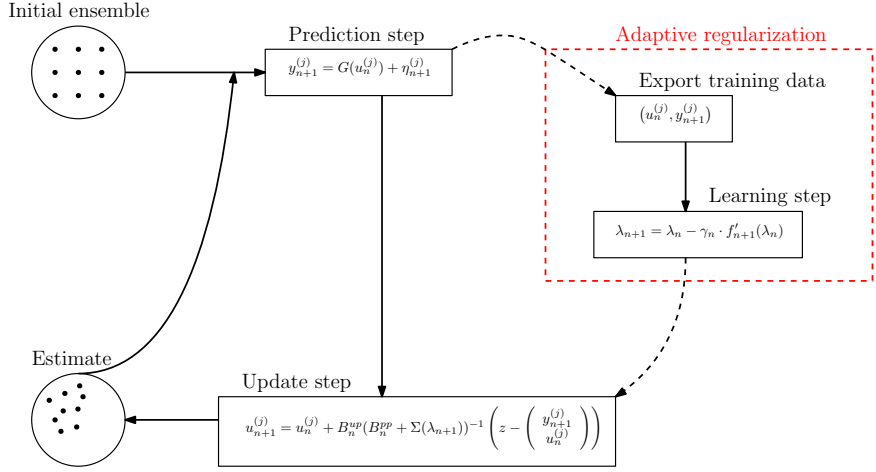}
\caption{Representation of adaptive regularized ensemble Kalman inversion with the inclusion of data-driven learning.}
\label{fig:adaptive_reg_dd}
\end{figure} 

{
\begin{remark}
An important question to ask is how to choose the step size $\gamma_n$. It is well known in optimization that it can beneficial to choose a non-fixed decreasing step size for maximum learning. Our choice for $\gamma_n$ will be based on the Armijo rule to ensure that we have a correct descent direction at every iteration.
\end{remark}
\begin{remark}
In the case of a linear forward model $G(\cdot) = A\cdot$, the resulting linearization is exact with $A_{\rm{apprx}} = A$ and $a_{\rm{apprx}}=0$.  Furthermore, we note that the linearization is only applied to learn the regularization parameter, while the TEKI update remains in the nonlinear setting. Applying a linearization in order to update the regularization parameter might be sufficient whereas the original inverse problem needs to be treated more carefully for nonlinear forward models.
\end{remark}
}

\subsection{MAP formulation}\label{ssec:MAP}

Another more simplistic way to find the parameter is through the Bayesian framework. Given some initial $\lambda$ which we can define through a prior of the form $\lambda\sim \mathcal{U} [0,M]$, then the parameter estimation is defined as 
\[
\argmax_{u\in\R^{d_u} ,\lambda\in(0,M)}\frac{1}{\sqrt{ \det(2\pi \Gamma) }}\exp(-\frac12\|G(u)-y\|_\Gamma^2)  \frac{1}{\sqrt{ \det(2\pi \lambda^{-1} C_0) }}\exp(-\frac12\|u\|_{\lambda^{-1}C_0}^2).
\]
Or alternatively this can be viewed as taking the logarithm and ignoring the constants
\[
\argmin_{u\in\R^{d_u} ,\lambda\in(0,M)} \frac12\|G(u)-y\|_\Gamma^2+\frac\lambda2\|u\|_{C_0}^2-\frac{d_u}{2}\log \lambda,
\]
where $d_u$ denotes the dimension of $u$. Notice that when $u$ is given, the minimizer of $\lambda$ is explicitly found using critical point
\[
\lambda_\ast=\left(\frac1{d_u} \|u\|^2_{C_0}\right)^{-1}. 
\]
Viewing each update step of the Tikhonov EKI as step into direction of the MAP estimator, leads to the following update
\[
\lambda_{n+1}=\left(\frac1{d_u}\|\bar{u}_n\|^2_{C_0}\right)^{-1},\quad \text{or}\quad \lambda_{n+1}=\left(\frac1{Jd_u}\sum_{j=1}^J\|u^{(j)}_n\|^2_{C_0}\right)^{-1}. 
\]
{It can occur, that $\lambda$ will eventually go out the feasible set $[0,M]$. Therefore, in order for it to remain, we will introduce a projection operator
$\mathcal{P}_{[0,M]}: \R \to (0,M]$ that ensures $\lambda$ goes back in the feasible set. }

\begin{algorithm}[H]
\caption{Tikhonov EKI: adaptive regularization using the MAP}
\label{alg:TEKIadapt_MAP}
\begin{algorithmic}[1]
\State \textbf{Input} $\{u^{(j)}_0\}_{j=1}^{J} \sim \mathcal{N}(0,C_0)$,  $J\ge2$, $\lambda_0>0$

\For{$n=0,\ldots,N-1$},
\State compute
$$
\lambda_{n+1}=\mathcal P_{[0,M]}\left(\left(\frac1{Jd_u}\sum_{j=1}^J\|u^{(j)}_n\|^2_{C_0}\right)^{-1}\right).
$$
\State update the ensemble of particle by TEKI 
$$
u^{(j)}_{n+1} = u^{(j)}_n + B^{up}_{n} \big(B^{pp}_n + \Sigma({\lambda}_{n+1})\big)^{-1}\bigg(z - \begin{pmatrix} y_{n+1}^{(j)} \\ u_n^{(j)}\end{pmatrix}\bigg).
$$
\EndFor
\end{algorithmic}
\end{algorithm}

\subsection{Adaptive covariance EKI}
Our final adaptive method that we introduce in this paper follows closely to the ideas of hierarchical EKI \cite{NKC18,CIRS18}, where {we estimate hyperparameters simultaneously}. To incorporate regularization as prior information into EKI we proceed as follows. Working in a Bayesian setting,  we view $(u,y)$ as jointly distributed random variable, where the solution of the Bayesian inverse problem is given by the posterior distribution 
$$u\mid y\sim \mu({\mathrm d}u) = \frac{1}{Z} \exp(-\frac12\|G(u)-y\|_\Gamma^2)\mu_0({\mathrm d} u).$$ 
Here, $Z=\int_{\R^{d_u} } \exp(-\frac12\|G(u)-y\|_\Gamma^2)\mu_0({\mathrm d} u)$ denotes a normalization constant and $\mu_0$ is the prior distribution. Assuming that $\mu_0$ depends on a hyperparameter $\theta$ gives the possibility to tune the resulting estimate of the Bayesian inverse problem.  Suppose that $\mu_0$ can be represented as Lebesgue density $\mu_0(u,\theta)$, and we have access to prior information on $\theta$ given by the pdf $q_0$, then we are able to compute the posterior distribution w.r.t.~$u$ as well as $\theta$ by 
$$ (u,\theta)\mid y \sim \mu({\mathrm d}(u,\lambda)) = \frac{1}{Z} \exp(-\frac12\|G(u)-y\|_\Gamma^2)\mu_0(u,\theta)q_0(\theta)\,{\mathrm d}(u,\theta).$$
In case we assume a Gaussian prior distribution on $u$ given by
$$\mu_0(u,\lambda) = \frac{1}{\sqrt{\det(2\pi C_0(\theta))}}\exp(-\frac12\|u-m_0(\theta)\|_{C_0(\theta)}^2),$$
we can treat the arising hyperparameters $\theta$ as parameters for choosing the prior mean $m_0(\theta)$ and covariance matrix $C_0(\theta)$. The presented MAP formulation in the previous Section~\ref{ssec:MAP} can be viewed as special case $C_0(\theta)\mapsto \lambda^{-1}C_0$, where $\theta=\lambda\sim\mathcal U((0,M))$ is the only hyperparameter to learn. Recall that in this setting the MAP estimate w.r.t.~$(u,\lambda)$ is given by
\[\argmin_{u\in\R^{d_u} ,\lambda\in(0,M)}\ \frac12\|G(u)-y\|_\Gamma^2+\frac\lambda2\|u\|_{C_0}^2-\frac{d_u}{2}\log \lambda,\]
for which we now propose a two-level optimization method. The first level concerns with the minimization w.r.t.~$u$ applying TEKI as preconditioned gradient descent, while the second level with the minimization w.r.t.~$\lambda$ applying gradient descent. Therefore, we consider the two-level update scheme on the particle system $(u_n^{(j)},\lambda_n^{(j)})_{j=1,\dots,J}$:
\begin{align*}
u^{(j)}_{n+1} &= u^{(j)}_n + B^{up}_{n} \big(B^{pp}_n + \Sigma({\lambda}_{n}^{(j)})\big)^{-1}\bigg(z - \begin{pmatrix} y_{n+1}^{(j)} \\ u_n^{(j)}\end{pmatrix}\bigg),\\
\lambda^{(j)}_{n+1} &= \lambda_n^{(j)} - \frac12\|u_n^{(j)}\|_{C_0}^2+\frac{d_u}2\frac{1}{\lambda_n^{(j)}}.
\end{align*}
{We note that from a conceptual point of view, there is no advantage of applying this two-level scheme over the previously presented Algorithm~\ref{alg:TEKIadapt_MAP}. However, the use of the MAP estimate for $\theta$ (based on the current estimate for the unknown $u$) while being still far away from the true parameters might results in instabilities. Furthermore, the two-level scheme can be straightforwardly generalized to the case of learning the whole covariance structure (see below) or can even be used in a much more general setting with non-Gaussian priors, where no closed form solution for the optimal hyperparameters is available.} Assuming that the prior covariance is given by $C_0(\theta)$ for hyperparameters $\theta\in\R^{{d_u}_\theta}$ with uniform prior assumption $\theta\sim\mathcal U((0,M)^{{d_u}_\theta})$, the MAP estimate w.r.t.~$(u,\theta)$ is given by
\begin{equation}
\label{eq:minn}
\argmin_{u\in\R^{d_u},\theta\in(0,M)^{{d_u}_\theta}}\ \frac12\|G(u)-y\|_\Gamma^2+\frac12\|u\|_{C_0(\theta)}^2+\frac{1}{2}\log(\det(C_0(\theta))),
\end{equation}
and our two-level update scheme for the particle system $(u_n^{(j)},\theta_n^{(j)})_{j=1,\dots,J}$ is given by
\begin{align}
\label{eq:algo3_1}
u^{(j)}_{n+1} &= u^{(j)}_n + B^{up}_{n} \big(B^{pp}_n + \Sigma(C({\theta}_{n}^{(j)}))\big)^{-1}\bigg(z - \begin{pmatrix} y_{n+1}^{(j)} \\ u_n^{(j)}\end{pmatrix}\bigg),\\
\label{eq:algo3_2}
\theta^{(j)}_{n+1} &= \theta_n^{(j)} - \nabla_\theta\left(\frac12\|u_n^{(j)}\|_{C_0(\theta_n^{(j)})}^2\right)-\nabla_\theta\left(\frac{1}{2}\log(\det(C_0(\theta_n^{(j)}))\right).
\end{align}
{A natural extension of this approach is to learn the whole covariance structure instead of just one parameter corresponding to the scaling of $C_0$, i.e. we are in the following interested in learning (the reciprocals of) the eigenvalues of the prior covariance, which corresponds to learning the weight of the individual dimensions. Assuming an eigen-decomposition of $C_0$ of the form
\begin{equation}\label{eq:eigCov}
C_0=UD(\theta)
U^\top,\quad D(\theta)=\begin{pmatrix} 1/\theta_1 & & \\
&\ddots&\\
& & 1/\theta_{d_u}
\end{pmatrix},
\end{equation}
with orthonormal matrix $U\in\mathbb \R^{d_u \times d_u}$. 
To simplify notation, we assume w.l.o.g. that $C_0$ is a diagonal matrix. This can be always satisfied by reparametrizing the problem in the eigenbasis. The MAP estimate then solves the problem
\begin{equation*}
\argmin_{u\in\R^{d_u},\theta\in \R^{^{{d_u}_\theta}}}\ \frac12\|G(u)-y\|_\Gamma^2+\frac12\langle U^\top u, {(D(\theta))^{-1}}U^\top u\rangle -\frac{1}{2}\log(\prod_{j=1}^{d_u}{\theta_j})\,.
\end{equation*}

In the first level we again apply TEKI, while in the second level we do gradient descent w.r.t to the reciprocals of the eigenvalues of $C_0$ based on the empirical mean of the parameters $u$. Hence, our proposed adaptive algorithm proceeds as follows.  
\begin{algorithm}[H]
\caption{Tikhonov EKI: adaptive regularization via learning the covariance operator}
\label{alg:TEKIadapt_cov}
\begin{algorithmic}[1]
\State \textbf{Input} $\{u^{(j)}_0\}_{j=1}^{J} \sim \mathcal{N}(0,C_0)$,  $J\ge2$, $(\theta_0)_1,\ldots,(\theta_0)_{d_u}>0$
\State set
\[
\Sigma_{0}  =
\begin{bmatrix}
\Gamma & 0\\
0 &   U D(\theta_0) U^\top
\end{bmatrix}.
\]
\For{$n=0,\ldots,N-1$}
\For{$k=1,\ldots,d_u$}
\State transform $v_n^{(j)} = U^\top u_n^{(j)}$, $\bar v_n = \frac1J\sum_{j=1}^J v_n^{(j)}$,
\State update the eigenvalues by
$$(\theta_{n+1})_k = (\theta_n)_k - \frac12(\bar v_n)_k^2+\frac{1}2\frac{1}{(\theta_n)_k},$$
\State set
\[\Sigma_{n+1}  =
\begin{bmatrix}
\Gamma & 0\\
0 & U D(\theta_{n+1}) U^\top
\end{bmatrix}.\]
\EndFor
\State  update the ensemble of particle by TEKI 
$$
u^{(j)}_{n+1} = u^{(j)}_n + B^{up}_{n} \big(B^{pp}_n + \Sigma_n\big)^{-1}\bigg(z - \begin{pmatrix} y_{n+1}^{(j)} \\ u_n^{(j)}\end{pmatrix}\bigg).
$$
\EndFor
\end{algorithmic}
\end{algorithm}
}

 \subsection{Comparison}
An important question, related to above algorithms, is how they compare and what one should expect in practice. Before exploring this in the succeeding section, we note that Algorithm \ref{alg:TEKIadapt_nonlinear} is derived from
\cite{CSTW20,CE17} which has existing theory verifying this form of regularization. Therefore we expect it to perform well, especially in the linear setting. However this theory does not directly apply to the nonlinear
setting as there is no closed form for the Tikhonov solution, and thus we can not expect the same gains over the other algorithms.  In this case,  we have proposed to use a rough linear approximation on the forward model in order to choose the regularization parameter, but applying TEKI in order to solve the inverse problem through the Tikhonov regularized optimization problem in the nonlinear setting .
For the other two methodologies, Algorithm \ref{alg:TEKIadapt_cov} is similar to Algorithm \ref{alg:TEKIadapt_MAP}, which is based on the MAP formulation. In particular,  the former can be viewed as a special case of the latter, therefore we would expect Algorithm~\ref{alg:TEKIadapt_cov} to perform well, as we are updating an ensemble of particles based on the sample mean and covariances. While for Algorithm~\ref{alg:TEKIadapt_MAP} we consider the direct computation of stationary points for $\lambda$,  due to numerical stability we suggest to apply a gradient descent method for learning the covariance matrix in Algorithm~\ref{alg:TEKIadapt_cov}.
The relevance of learning the whole covariance matrix is illustrated in our nonlinear numerical example, where we estimate the coefficients of \eqref{eq:kl} such that the regularization covariance matrix describes the impact of different coefficients.
\section{Numerical examples}
\label{sec:num}

In this section we numerically test and implement the adaptive strategies discussed in Section \ref{sec:adapt}. As our analysis holds in the linear case, we will test our algorithms on a linear partial differential equation (PDE). To gain further insight we also test our algorithms on a non-linear problem arising from geophysical sciences, that of Darcy flow. We highlight the effect of Tikhonov regularization within EKI with noisy observations and the efficient improvement through our adaptive strategies. 

\subsection{Linear partial differential equation}
Our first set of experiments is to show, with the help an inverse elliptic PDE, that Tikhonov regularization and in particular our presented adaptive strategies improves the stability of EKI as iterative solver for the inverse problem. Throughout our experiments we are interesting in assessing the performance of the TEKI in the noisy case through:
\begin{enumerate}
\item[(i)] Data misfit: $\mathbb{E}\big[\frac1J\sum\limits_{j=1}^J |Au^{(j)} - y |^2\big]$.
\item[(ii)] Tikhonov loss function: $\mathbb{E}\big[\frac{1}{J}\sum\limits_{j=1}^J\mathcal{I}(u^{(j)},\lambda)\big]$. 
\item[(ii)] Residual: $\mathbb{E}\big[\frac1J\sum\limits_{j=1}^J |\widetilde r^{(j)}|^2\big]$. \\
\end{enumerate}

We will use \eqref{eq:updateT} as discretization method of the continuum limit of the algorithm in form of \eqref{eq:TEKI_SDE}.  Our forward model will be a linear 1D elliptic PDE of the form, where we seek a solution 
 $p \in  \mathcal{U}:= H^1_0(D)$ from 
\begin{align}
\label{eq:fwd1}
\frac{d^2p}{dx^2} + p &= u, \ \ \ x \in D, \\
\label{eq:bc1}
p &= 0,  \ \ \ x \in \partial D.
\end{align}
The inverse problem associated with \eqref{eq:fwd1} is the recovery of $u \in \mathcal{X}=L^\infty(D)$ from $K=8$ pointwise measurements of $p$. Our forward solver for \eqref{eq:fwd1} is a piecewise finite element method with mesh size $h=2^{-4}$ over the domain $D=(0,\pi)$. Thus, the forward map $\mathcal{G}(\cdot) = A \cdot$ is linear, where we set $A = \mathcal{O} \circ G$, where $G:\mathcal{X} \rightarrow \mathcal{U}$ is the solution operator and $\mathcal{O}:\mathcal{U} \rightarrow \R^K$ is the observational operator taking measurements at $K$ equidistantly chosen points in $D$, i.e.~$\mathcal O(p) = (p(x_1),\dots,p(x_K))^\top$. We specify the covariance of the noise as $\Gamma = \gamma^2\cdot I$ where $\gamma = 0.1$ and consider the prior assumption 
$$u_0\sim \mathcal{N}(0,C_0),$$ 
where $C_0:=\lambda^\dagger\cdot10\cdot(-\Delta)^{-1}$.
For our numerical examples we will consider the true unknown parameter $$u^\dagger \sim \mathcal{N}(0,\frac{1}{\lambda^\dagger}\cdot C_0),$$
such that our aim will be to handle the difference between prior assumption and underlying ground truth.  For the variance inflation in all of our numerical results we choose the inflation factor to be $\alpha=1/2$ and $R=1$. 
As reference we compute the best possible approximation $T_{\lambda_{\textrm{best}}}(y)$ in the sense that \[\lambda_{\textrm{best}} = \underset{\lambda}\argmin \|T_{\lambda}(y)-u^\dagger\|^2.\]

For example, for a realized ground truth $u^\dagger$ with $\lambda^\dagger=50$, the following Figure~\ref{fig:ex1_regpar} shows the dependence of the residual between underlying ground truth and Tikhonov solution on the regularization parameter $\lambda$.
\begin{figure}[!htb]
	\includegraphics[width=0.5\textwidth]{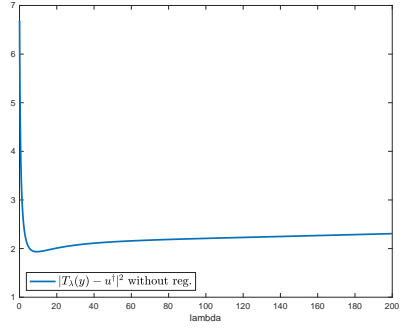}
    \caption{Difference of the Tikhonov minimizer $T_{\lambda}(y)$ to the ``known" unknown true parameter $u^\dagger$ for the case $\lambda^\dagger=50$.}\label{fig:ex1_regpar}
\end{figure} 

For the fixed regularization comparison, we will choose $\lambda=1$, i.e.~in this case we ``trust" the prior assumption. For each regularization algorithm we test two different examples which correspond to different values of $\lambda^\dagger$ used to generate the underlying ground truth. These will be chosen as $\lambda^\dagger \in \{0.04,\ 50\}$. We will keep the number of paths and particles consistent for each example and algorithm, specified as $Q=100$ paths of \eqref{eq:TEKI_SDE} with ensemble size $J=50$.

\subsubsection{Case $\lambda^\dagger=50$}
\begin{figure}[!htb]
	\begin{subfigure}[c]{0.49\textwidth}
	\includegraphics[width=1.1\textwidth]{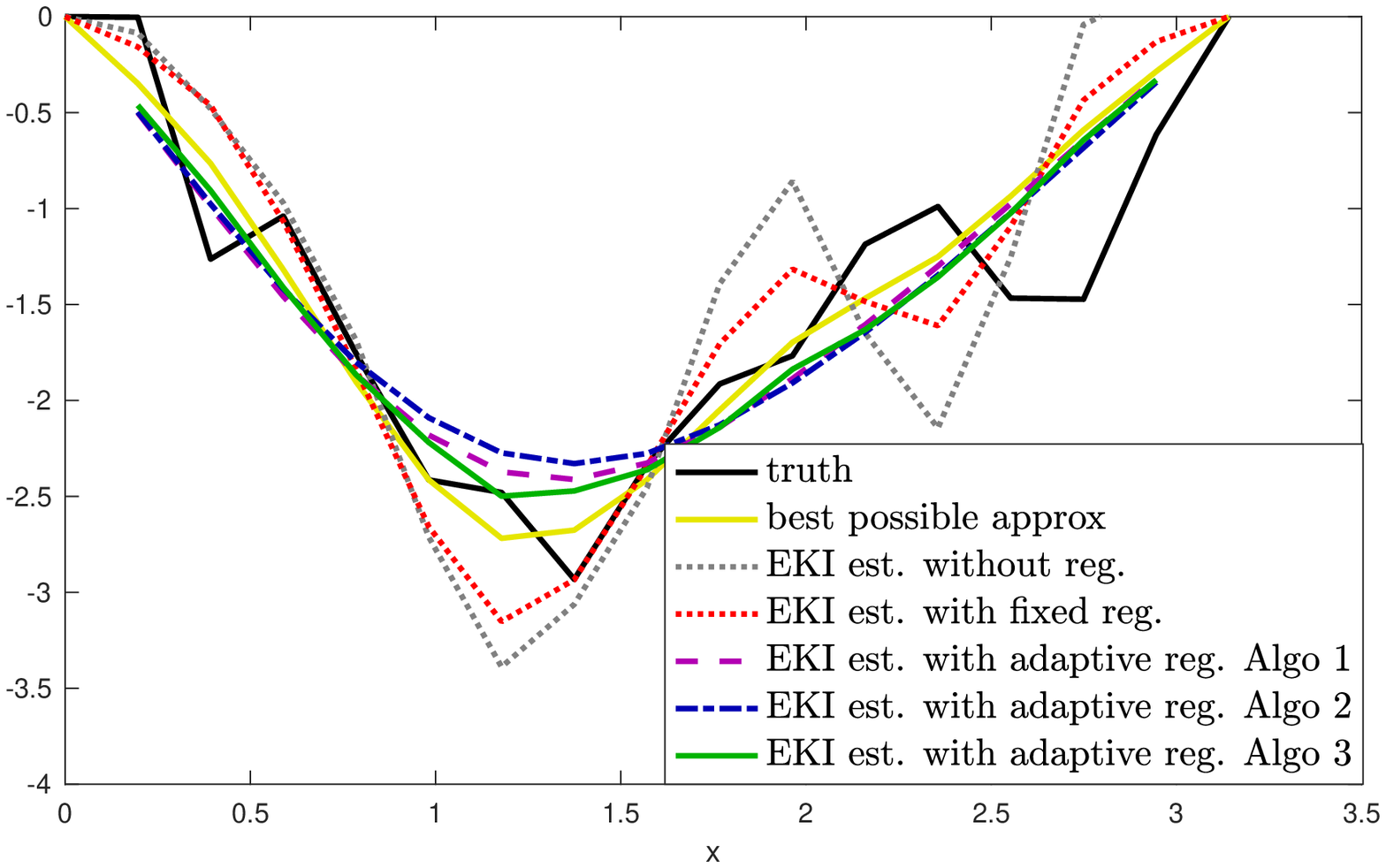}
	\end{subfigure}
	\begin{subfigure}[c]{0.49\textwidth}
	\includegraphics[width=1.1\textwidth]{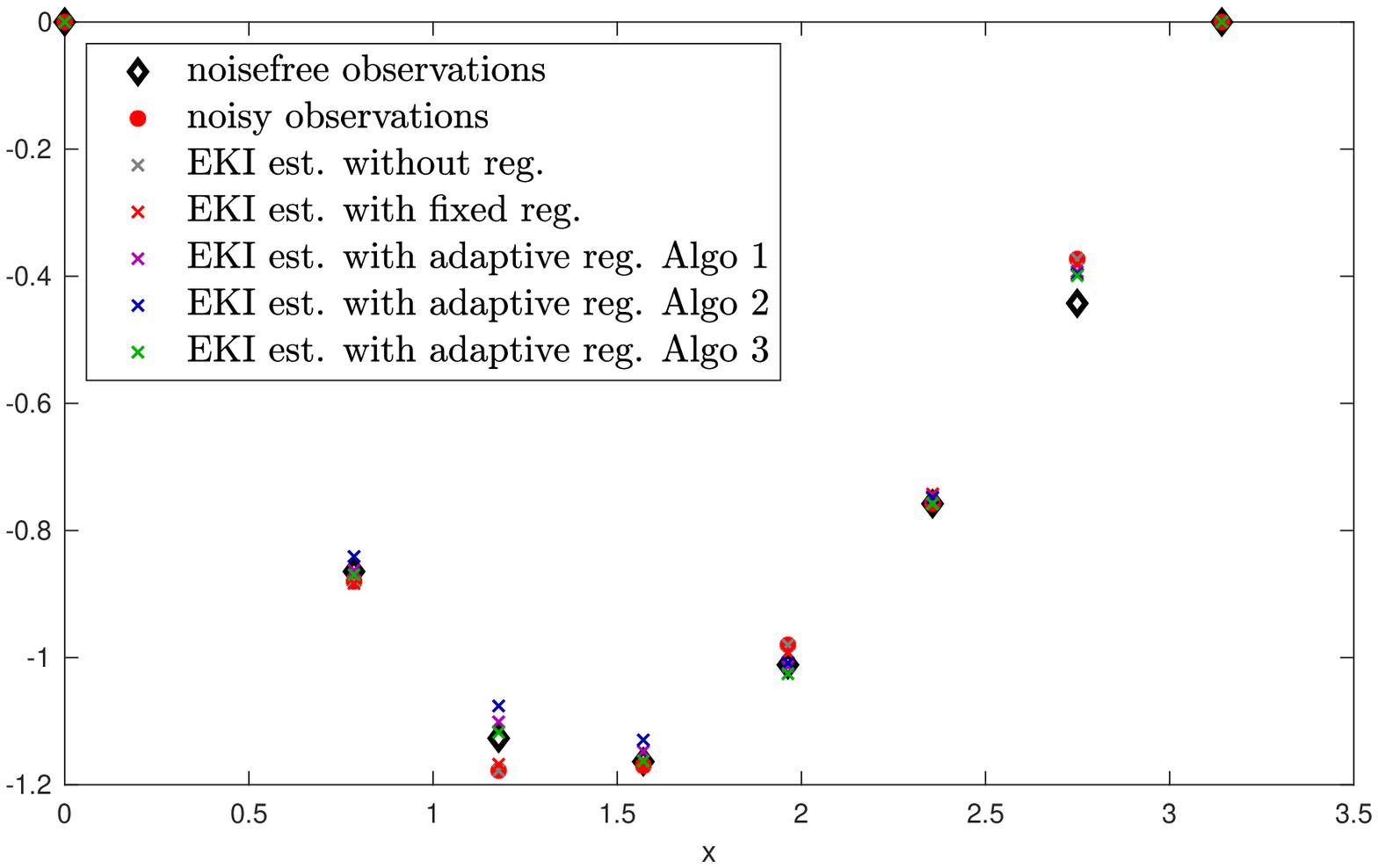}
	\end{subfigure}
    \caption{(T)EKI estimation of the unknown parameter (left) and the corresponding observations (right) for the different presented algorithms in the linear example with $\lambda^\dagger=50$.}\label{fig:ex1_est}
\end{figure} 
\begin{figure}[!htb]
	\begin{subfigure}[c]{0.49\textwidth}
	\includegraphics[width=1.1\textwidth]{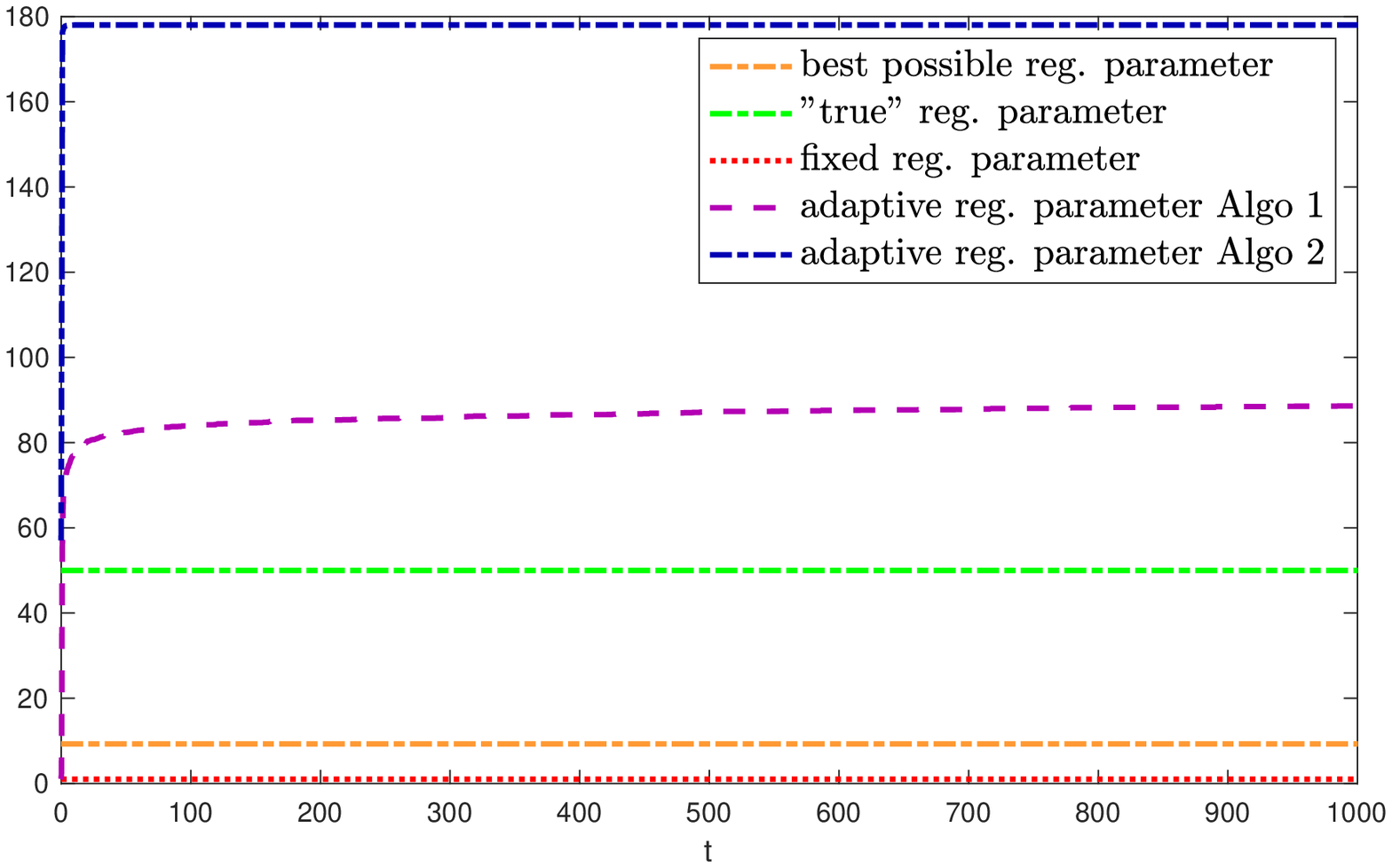}
	\end{subfigure}
	\begin{subfigure}[c]{0.49\textwidth}
	\includegraphics[width=1.1\textwidth]{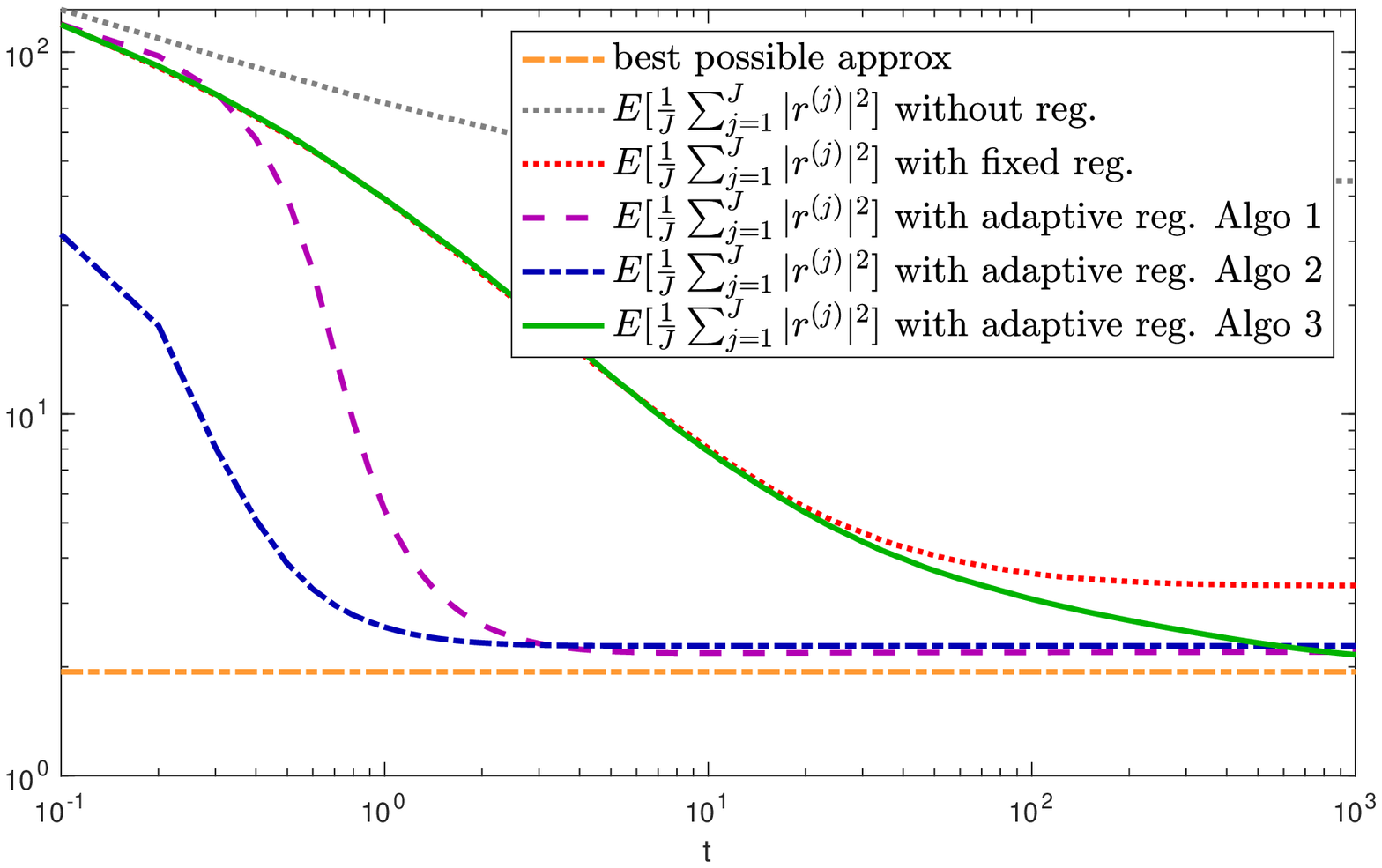}
	\end{subfigure}
    \caption{Learned regularization parameter (left) and the corresponding residuals for the different presented algorithms in the linear example with $\lambda^\dagger=50$.}\label{fig:ex1_regpar_res}
\end{figure} 
\begin{figure}[!htb]
	\begin{subfigure}[c]{0.49\textwidth}
	\includegraphics[width=1.1\textwidth]{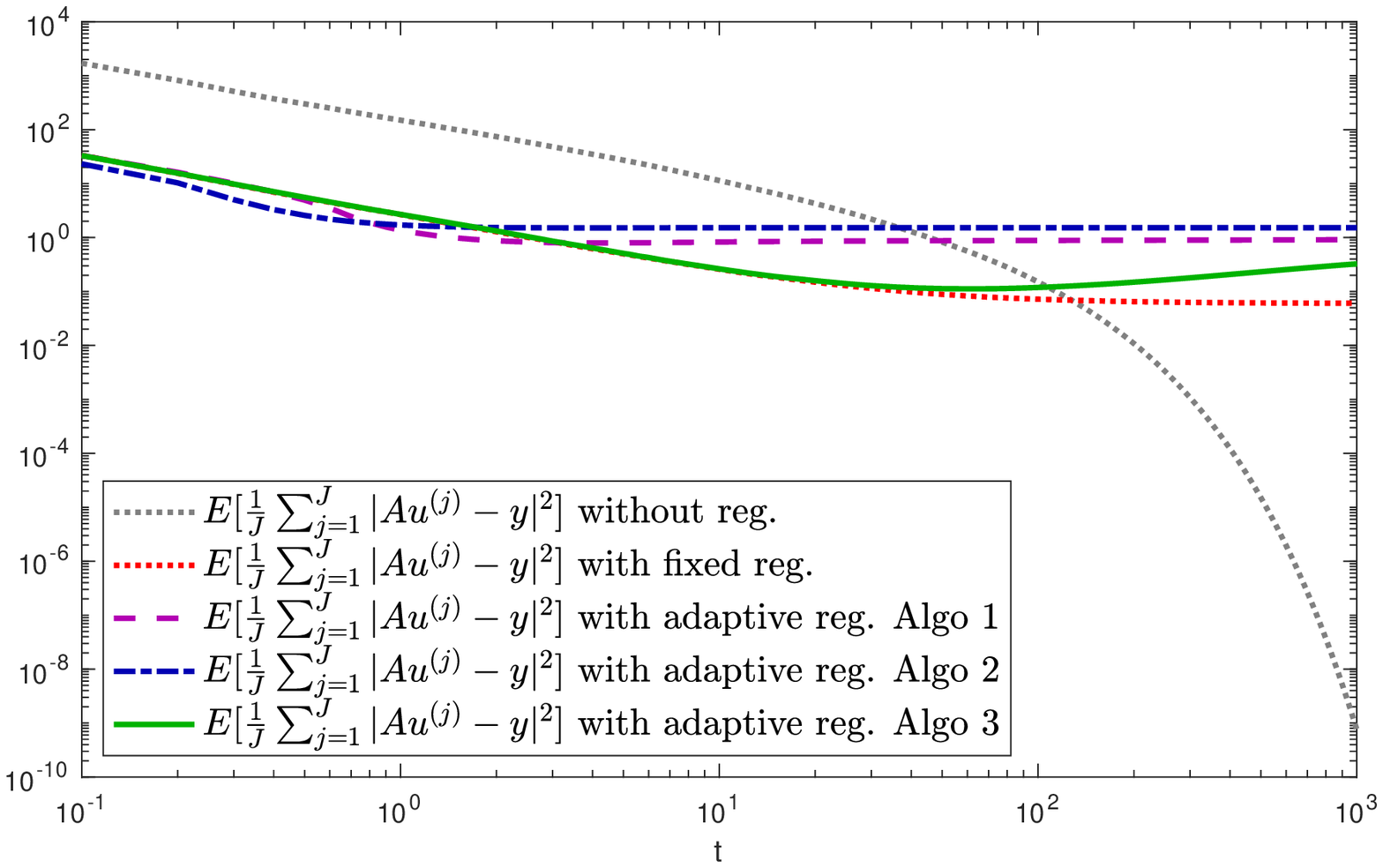}
	\end{subfigure}
	\begin{subfigure}[c]{0.49\textwidth}
	\includegraphics[width=1.1\textwidth]{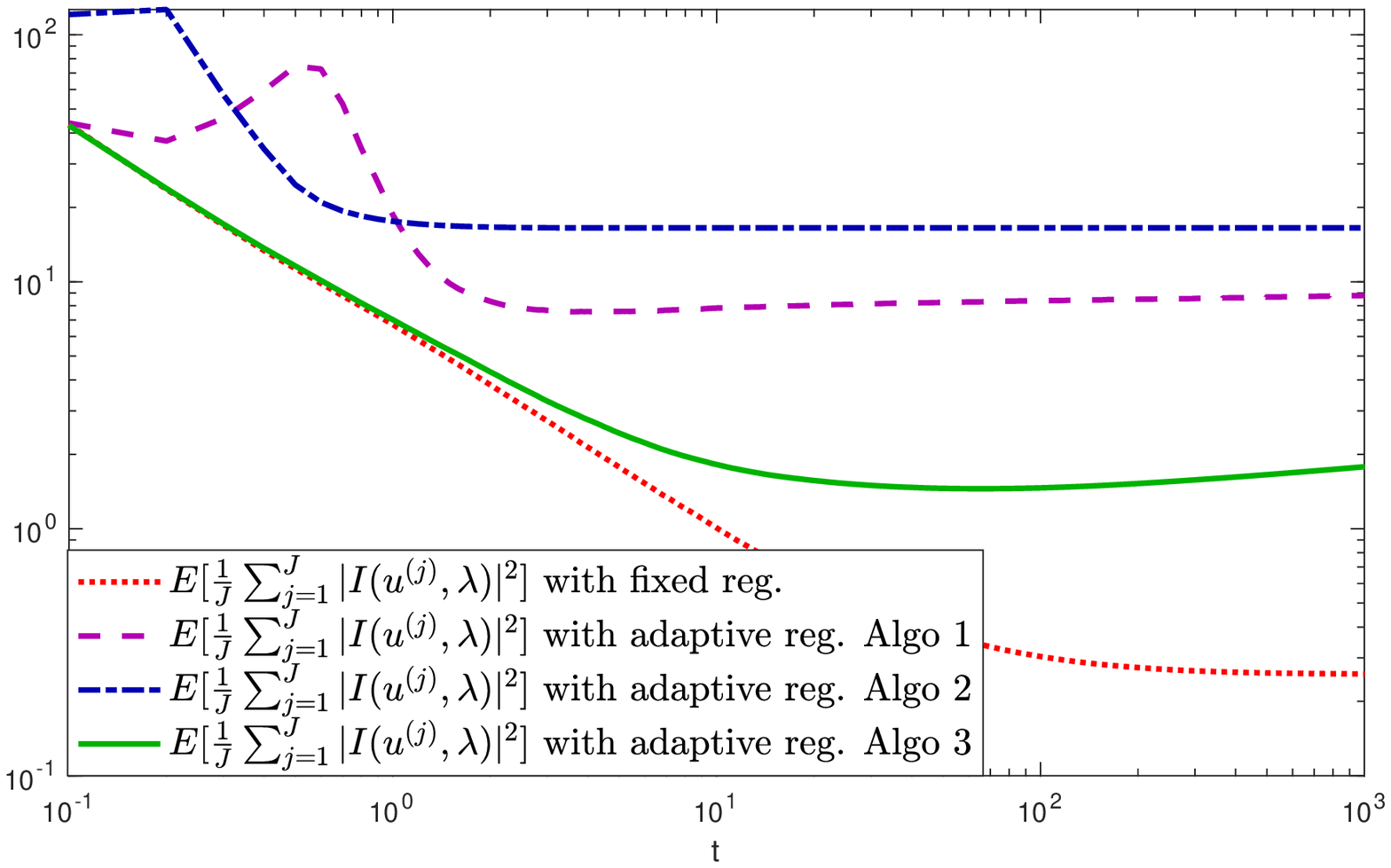}
	\end{subfigure}
    \caption{Data misfit (left) and Tikhonov regularized loss (right) for the different presented algorithms in the linear example with $\lambda^\dagger=50$.}\label{fig:ex1_loss}
\end{figure} 

\bigskip
Our first results from the numerics constitute to the choice ${\lambda}^\dagger=50$ in order to generate the underlying ground truth.  We compare the reconstruction of each algorithm to fixed and no regularization which is shown in Figure \ref{fig:ex1_est} w.r.t.~the parameter space as well as the resulting observations. As we can see EKI with no regularization performs the worst with most variation followed by using a fixed regularization.  The reason for this behaviour is that too much weight lies on the data misfit as a result overfitting of the data occurs, which can also be seen in Figure~\ref{fig:ex1_loss}. The three presented algorithms based on adaptively learning the regularization parameter prevent this overfitting issue and lead to a better approximation in the sense of smaller residuals, see Figure~\ref{fig:ex1_regpar_res}.  We note that the high values for the adaptive regularization parameters seem to be valid which can also be seen from the flat curve w.r.t.~$\lambda$ for the distance between Tikhonov solution and underlying ground truth in Figure~\ref{fig:ex1_regpar}.

\subsubsection{Case $\lambda^\dagger=0.04$}
\begin{figure}[!htb]
	\begin{subfigure}[c]{0.49\textwidth}
	\includegraphics[width=1.1\textwidth]{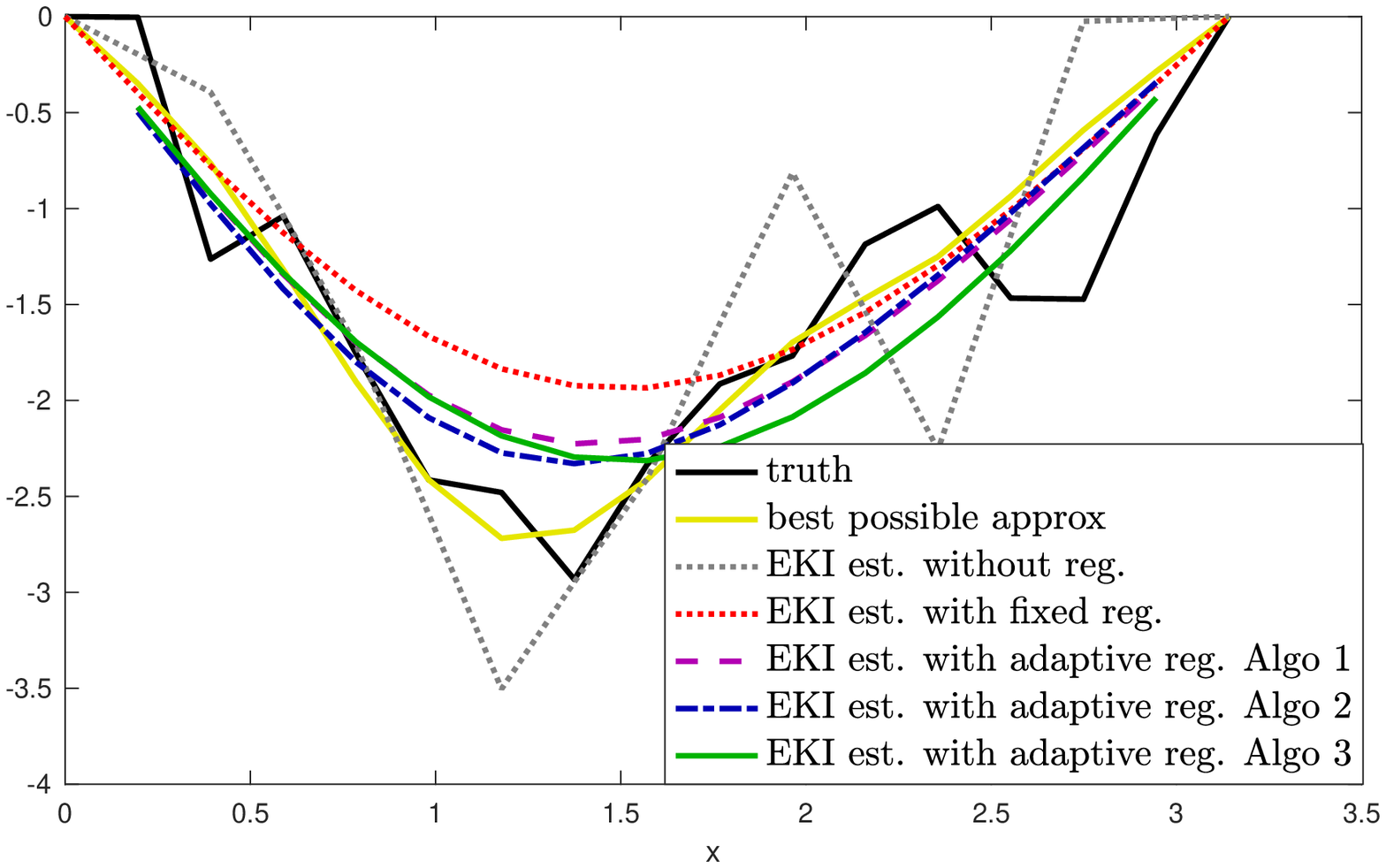}
	\end{subfigure}
	\begin{subfigure}[c]{0.49\textwidth}
	\includegraphics[width=1.1\textwidth]{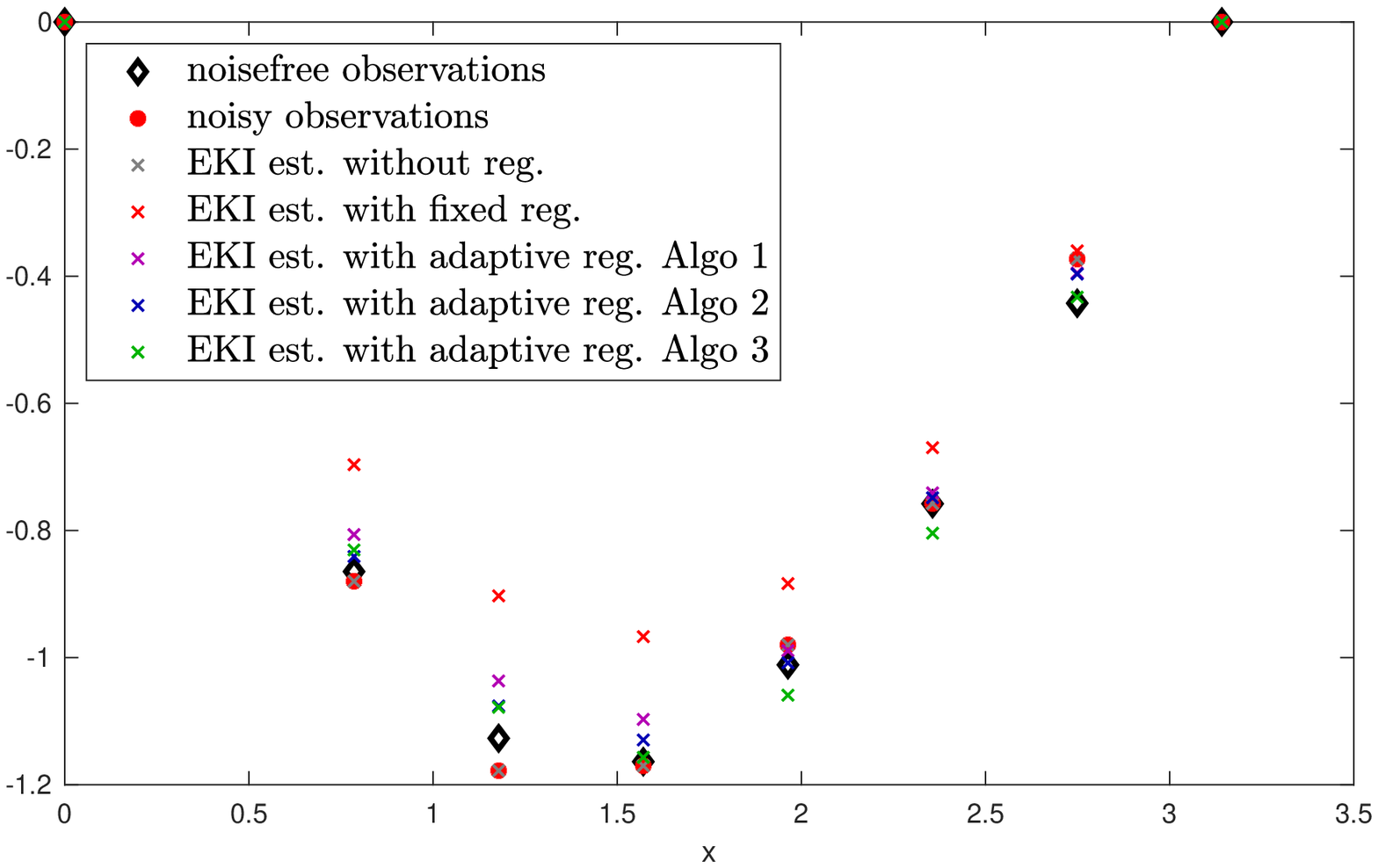}
	\end{subfigure}
    \caption{(T)EKI estimation of the unknown parameter (left) and the corresponding observations (right) for the different presented algorithms in the linear example with $\lambda^\dagger=0.04$.}\label{fig:ex2_est}
\end{figure} 
\begin{figure}[!htb]
	\begin{subfigure}[c]{0.49\textwidth}
	\includegraphics[width=1.1\textwidth]{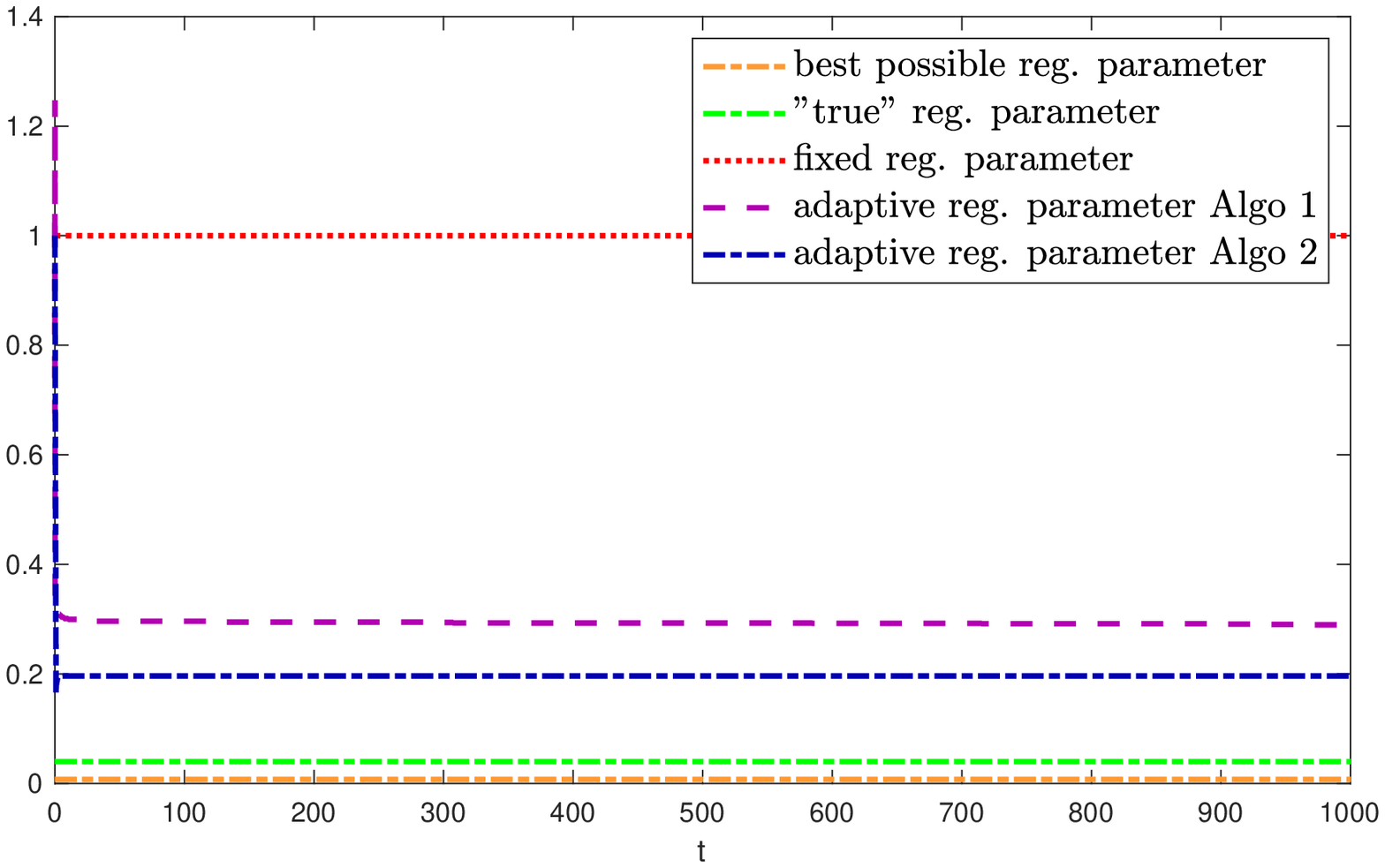}
	\end{subfigure}
	\begin{subfigure}[c]{0.49\textwidth}
	\includegraphics[width=1.1\textwidth]{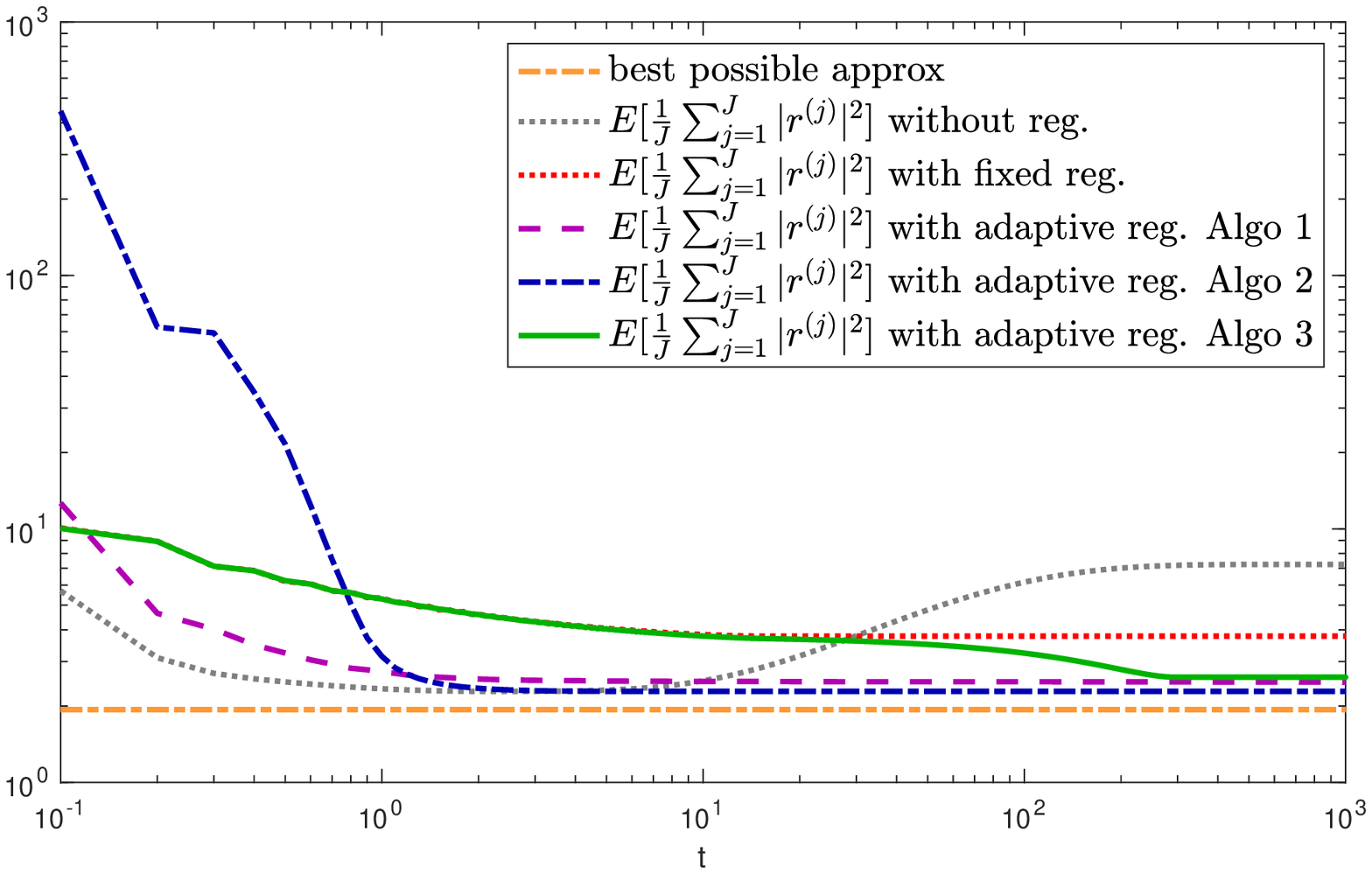}
	\end{subfigure}
    \caption{Learned regularization parameter (left) and the corresponding residuals for the different presented algorithms in the linear example with $\lambda^\dagger=0.04$.}\label{fig:ex2_regpar_res}
\end{figure} 
\begin{figure}[!htb]
	\begin{subfigure}[c]{0.49\textwidth}
	\includegraphics[width=1.1\textwidth]{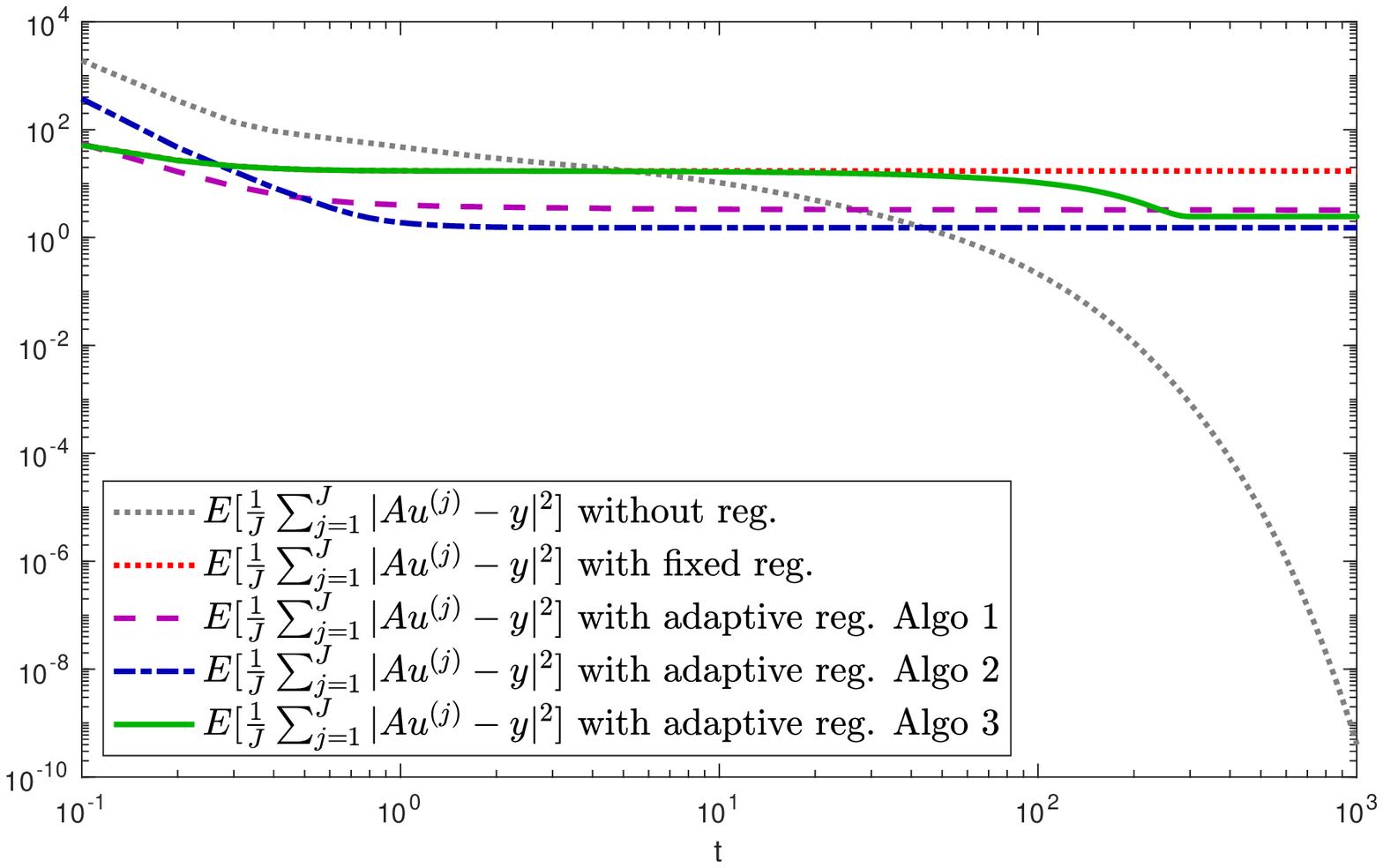}
	\end{subfigure}
	\begin{subfigure}[c]{0.49\textwidth}
	\includegraphics[width=1.1\textwidth]{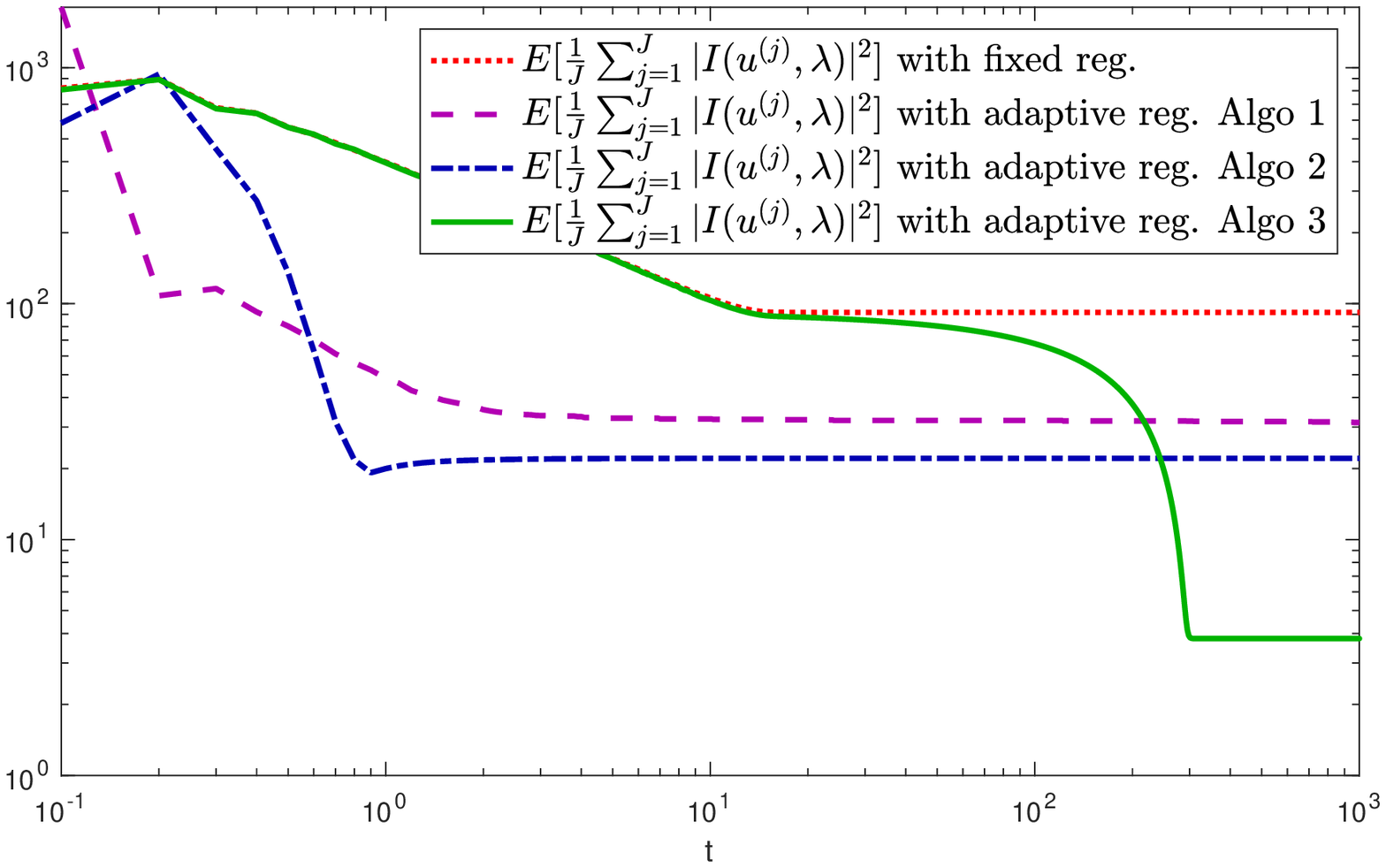}
	\end{subfigure}
    \caption{Data misfit (left) and Tikhonov regularized loss (right) for the different presented algorithms in the linear example with $\lambda^\dagger=0.04$. }\label{fig:ex2_loss}
\end{figure} 
For the second example we modify the true scaling parameter to be $\lambda^\dagger=0.04$.  While EKI without regularization still overfits the data, TEKI with fixed regularization parameter includes too much weight on the regularization and start to oversmooth. This can be seen in the resulting estimations in Figure~\ref{fig:ex2_est} as well as in the learned regularization parameters with corresponding residuals in Figure~\ref{fig:ex2_regpar_res}.
Again, all adapative algorithms perform similarly good and prevent the overfitting without including too much smoothness through regularization. This can also be seen in Figure~\ref{fig:ex2_loss} where we illustrate the data misfit and in particular the overfitting of EKI without regularization and the behavior of the Tikhonov regularized loss function. 
 
 \subsection{Nonlinear example: Darcy flow}
In the second problem,  we consider Darcys flow arising in geophysical sciences. The forward model models the pressure (or hydraulic head) given through the permeability. Mathematically given a source term $f$ and a permeability $\kappa \in L^{\infty}(D)$,  the forward problems is to solve the PDE 
\begin{align}
\label{eq:darcy1}
-\nabla\cdot({\kappa}\nabla p) &= f,   \ x \ \in \ D, \\
p&=0, \ x \  \in \partial D,
\end{align}
for $p \in \mathcal U:=H^1_0(D)$ which is subject to zero Dirichlet boundary conditions. In our implementation we consider a constant source term $f\equiv1$. We set a lognormal prior distribution, i.e.~$\kappa =\exp(u)$ with $u \sim \mathcal{N}(0,\frac{1}{\lambda^\dagger}C_0)$ where we define 
$$
C_0 := \sigma^2(\tau \cdot I - \Delta u)^{-\nu}.
$$
Here, $\sigma^2$ is a scaling constant, $\tau\ge0$ acts as a regularizing shift of the eigenvalues, $\nu > d/2$ is the smoothness of the prior and $\Delta$ is the Laplace operator in 1D. The nonlinear forward map is defined by $\mathcal G(\cdot) = \mathcal{O} \circ G(\cdot)$, where $\mathcal{O}:\mathcal{U} \rightarrow \R^K$ again the observational operator taking measurements at equidistantly chosen points in $D$ and $G:\mathcal{X} \rightarrow \mathcal{U}$ is the solution operator of \eqref{eq:darcy1} which has been numerically approximated by a second-order finite difference method on a uniform mesh of size $h=2^{-6}$. We will simulate our prior through a Karhunen-Lo\`{e}ve expansion \cite{LPS14} of the form 
\begin{equation}
\label{eq:kl}
u^\xi(x) = \sum^{\infty}_{j=1} \sqrt{\frac{1}{\lambda^\dagger}\sigma_j}\xi_j \phi_j(x), \quad \xi \sim \mathcal{N}(0,I),
\end{equation}
where $({\sigma_j},{\phi_j})$ is the eigenbasis of the covariance operator $C_0$.  In order to solve the inverse problem, we consider the task of estimating the coefficients $\xi$,  see also \cite{CIRS18,GHLS19} for more details. Therefore, we truncate \eqref{eq:kl} up to $d$ and introduce the nonlinear map $\mathcal G:\R^{d}\to\R^K$, with $\mathcal G(\xi) = \mathcal O \circ G(u^\xi(\cdot))$ and
$$u^\xi(\cdot) = \sum_{j=1}^d \xi_j\phi_j(\cdot).$$
Hence, our unknown parameter is given by $\xi\in\mathbb R^d$ with a Gaussian prior assumption $\mathcal N(0,\frac{1}{\lambda^\dagger}D_0)$, where 
\begin{equation*}
D_0 = \begin{pmatrix} \sigma_1 & \ & \ \\ \ & \ddots & \ \\ \ & \ & \sigma_d\end{pmatrix}, \quad
\sigma_j = \left(\frac{\sigma^2}{(j+\tau)^2}\right)^\nu,
\end{equation*}
where $\phi_j(x) = \sqrt{2\pi}\sin(2\pi x)$ are the eigenfunctions and $\lambda^\dagger>0$ is the true unknown.
We set $\lambda^\dagger=20$, $\tau=0$, $\sigma=1$, $\nu = 1$ and run  $Q=10$ paths of the TEKI with $J=50$ particles. The measurements noise is set to $\mathcal \mathcal{N}(0,\gamma^2\cdot I)$ with $\gamma=0.01$ and $K=16$ observation points. The physical domain $D=[0,1]$ has been discretized by the equidistant grid $\{\frac{i}{64},\ i=0,\dots,64\}$. For the fixed regularization we choose $\lambda=0.1$, we do not use variance inflation in our implementation and run the TEKI algorithm again based on the discrete version \eqref{eq:updateU}. 

The general trend of the numerics follow similarly to the linear PDE example. Specifically, we can observe again that all of the adaptive schemes lead to significant performance improvements.
This is seen through the reconstruction of the parameter in Figure \ref{fig:nonl_ex_est}.  Here, we can see again overfitting issues for EKI without regularization and TEKI with fixed but too small regularization parameter. This effect is further illustrated in Figure~\ref{fig:nonl_ex_regpar_res}, where we observe that the learned regularization parameters are larger than the fixed chosen one and the residuals improve through adaption of the regularization parameter, as well as in Figure~\ref{fig:nonl_ex_loss}, where we see that the adaptive schemes reduce the data misfit and include more weight on the penalization through Tikhonov regularization. Furthermore, we highlight the improvement of Algorithm~\ref{alg:TEKIadapt_cov} in Figure~\ref{fig:nonl_ex_eigenvalues_cov}, where we observe that the smallest eigenvalue of the regularization covariance matrix decreases much less than the largest eigenvalue 
and hence, leading to an improvement of the recovery due to the various weights w.r.t.~the coefficients in the KL expansion.  This effect can also be seen in the lower residual value compared to Algorithm~\ref{alg:TEKIadapt_nonlinear} and \ref{alg:TEKIadapt_MAP}, see Figure~\ref{fig:nonl_ex_regpar_res} (right).

\begin{figure}[!htb]
	\begin{subfigure}[c]{0.49\textwidth}
	\includegraphics[width=1.1\textwidth]{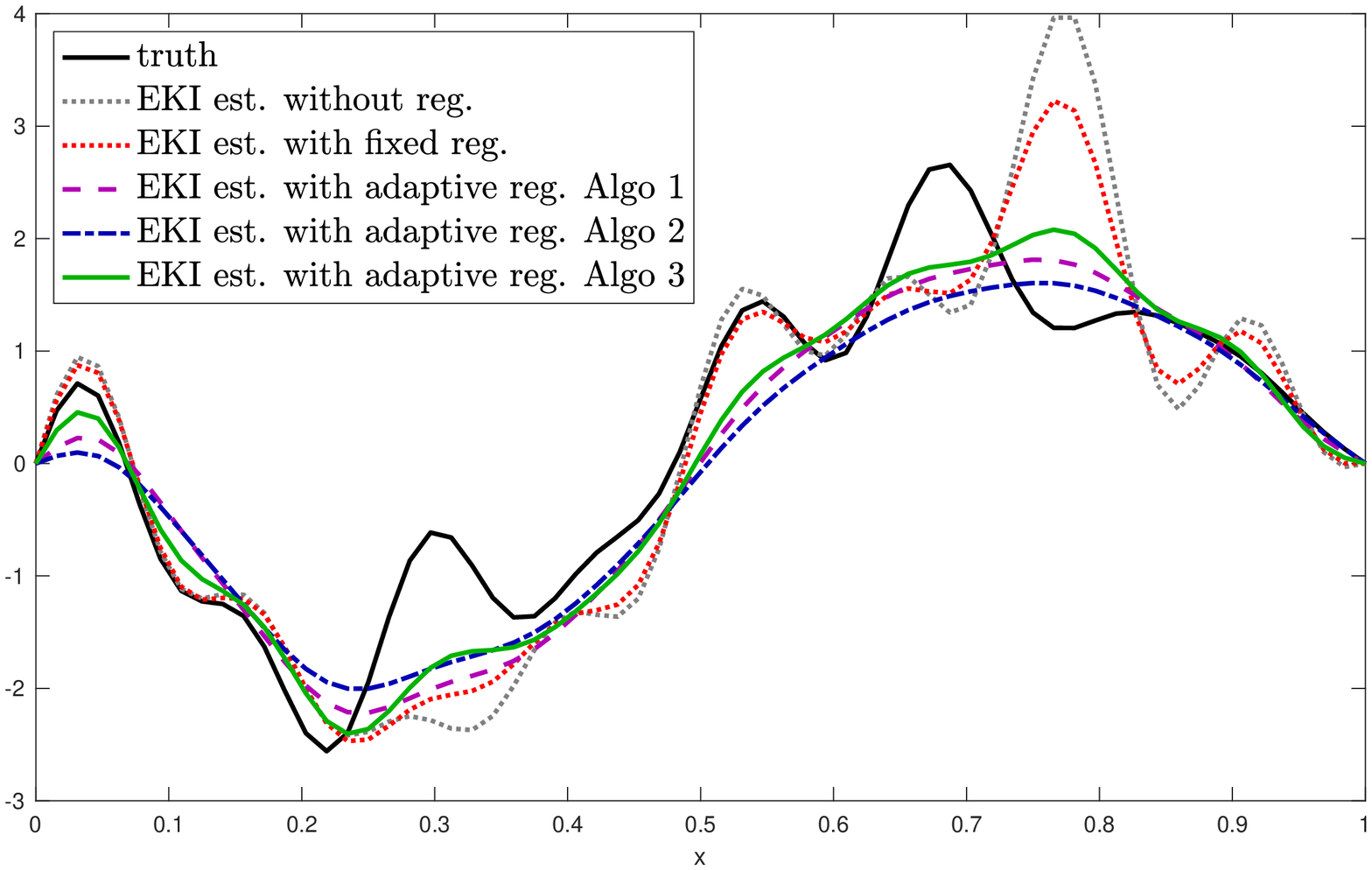}
	\end{subfigure}
	\begin{subfigure}[c]{0.49\textwidth}
	\includegraphics[width=1.1\textwidth]{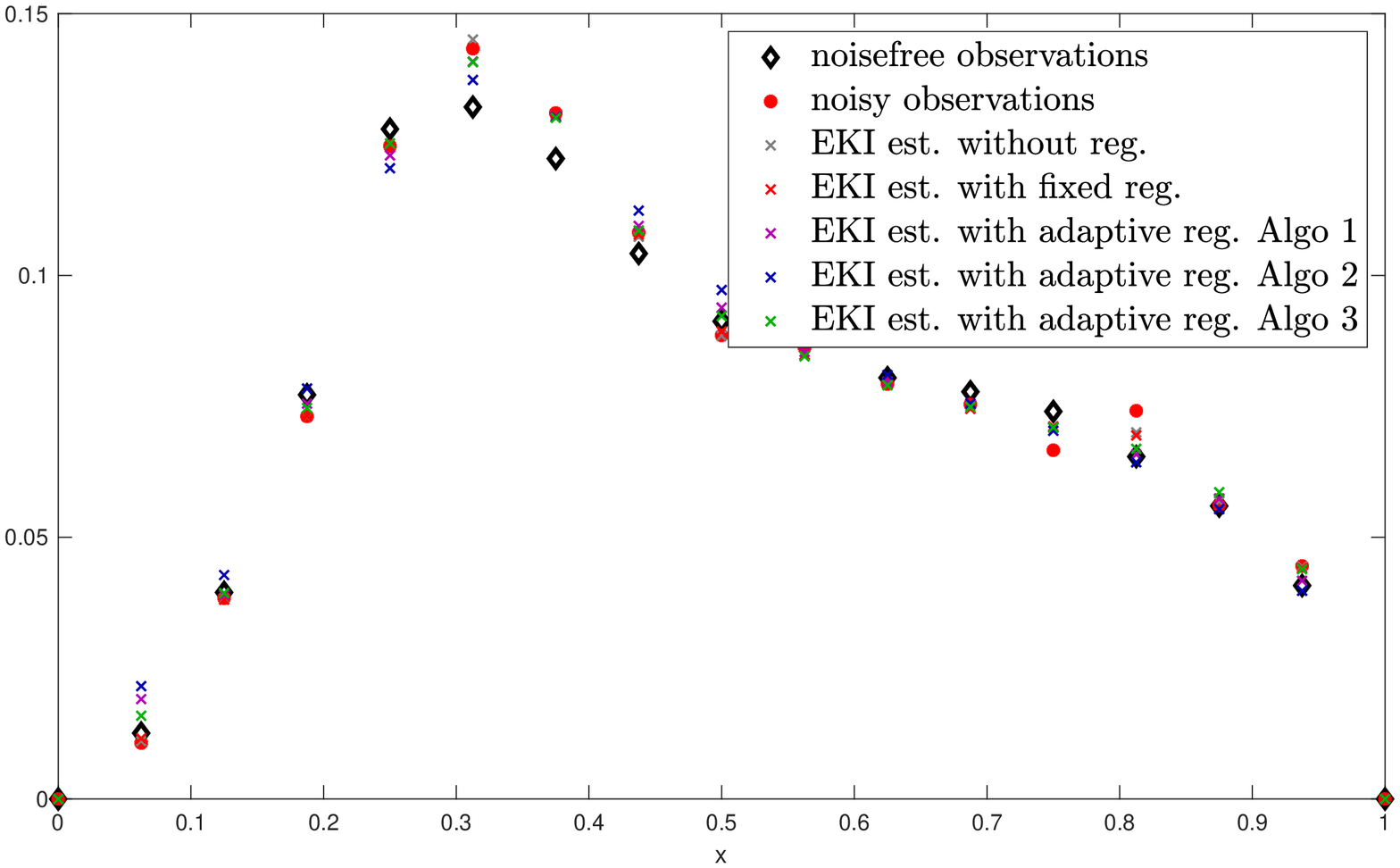}
	\end{subfigure}
    \caption{(T)EKI estimation of the unknown parameter (left) and the corresponding observations (right) for the different presented algorithms in the nonlinear example with $\lambda^\dagger = 20$.}\label{fig:nonl_ex_est}
\end{figure} 

\begin{figure}[!htb]
	\begin{subfigure}[c]{0.49\textwidth}
	\includegraphics[width=1.1\textwidth]{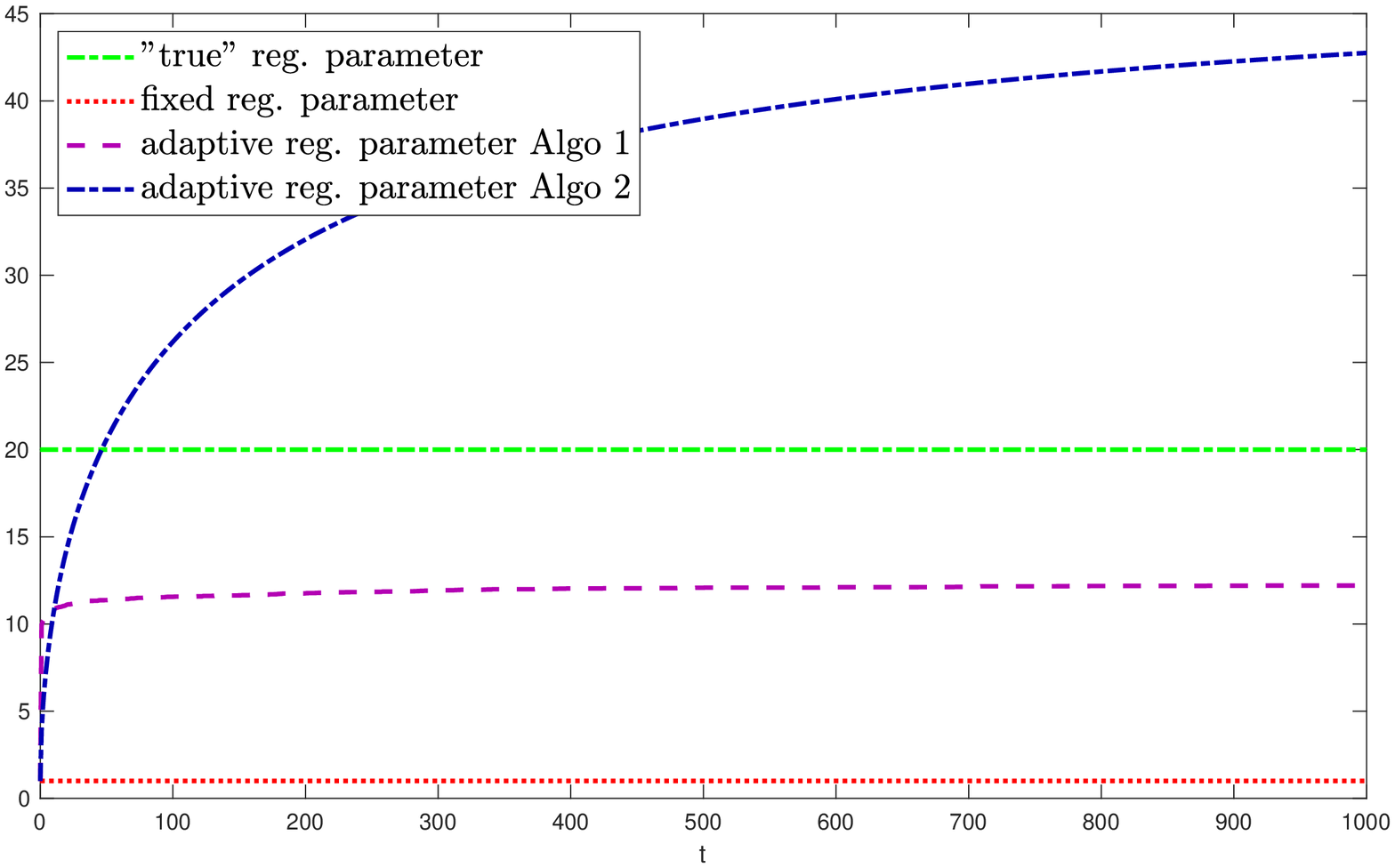}
	\end{subfigure}
	\begin{subfigure}[c]{0.49\textwidth}
	\includegraphics[width=1.1\textwidth]{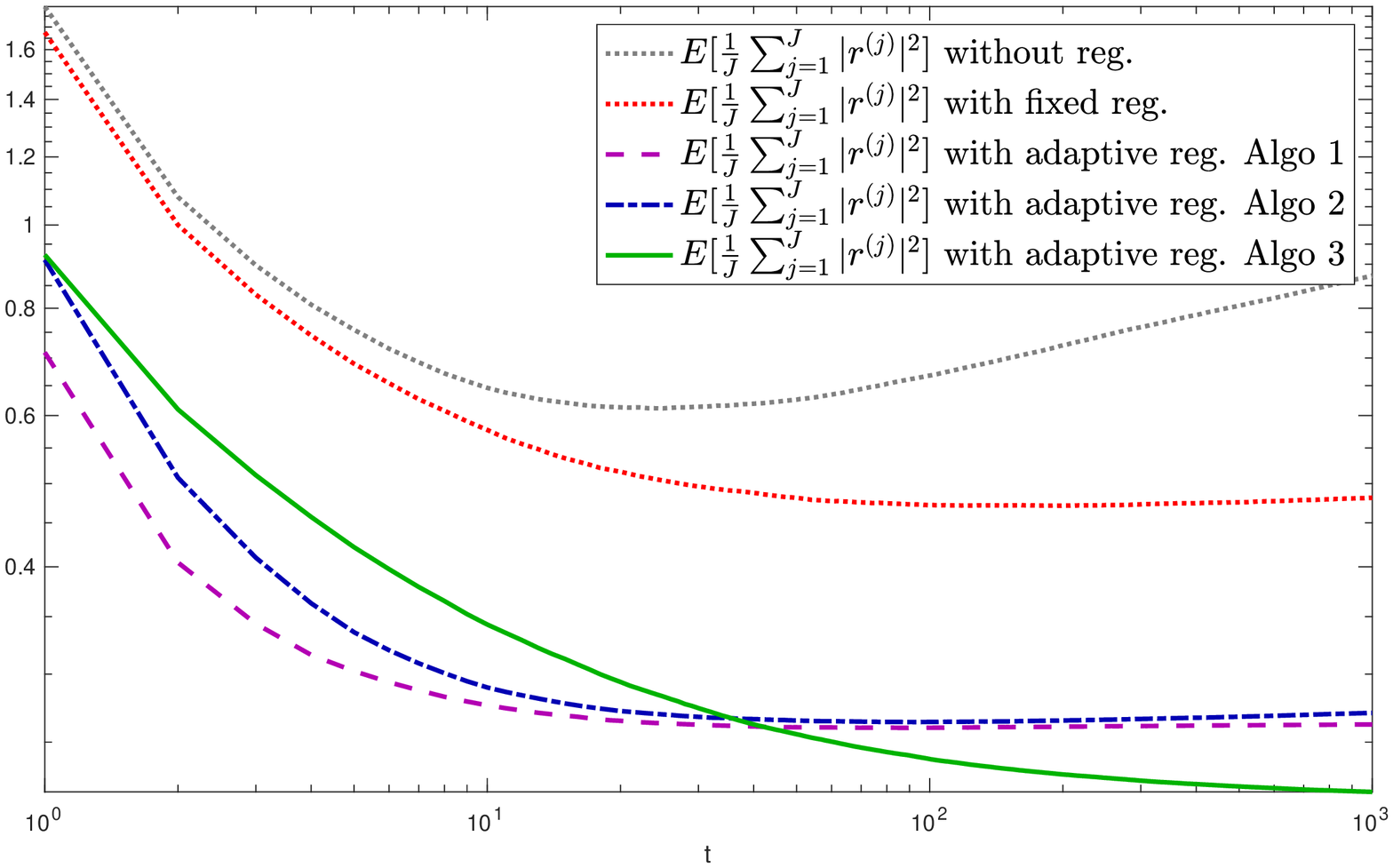}
	\end{subfigure}
   \caption{Learned regularization parameter (left) and the corresponding residuals for the different presented algorithms in the nonlinear example with $\lambda^\dagger = 20$.}\label{fig:nonl_ex_regpar_res}
\end{figure}

\begin{figure}[!htb]
	\begin{subfigure}[c]{0.49\textwidth}
	\includegraphics[width=1.1\textwidth]{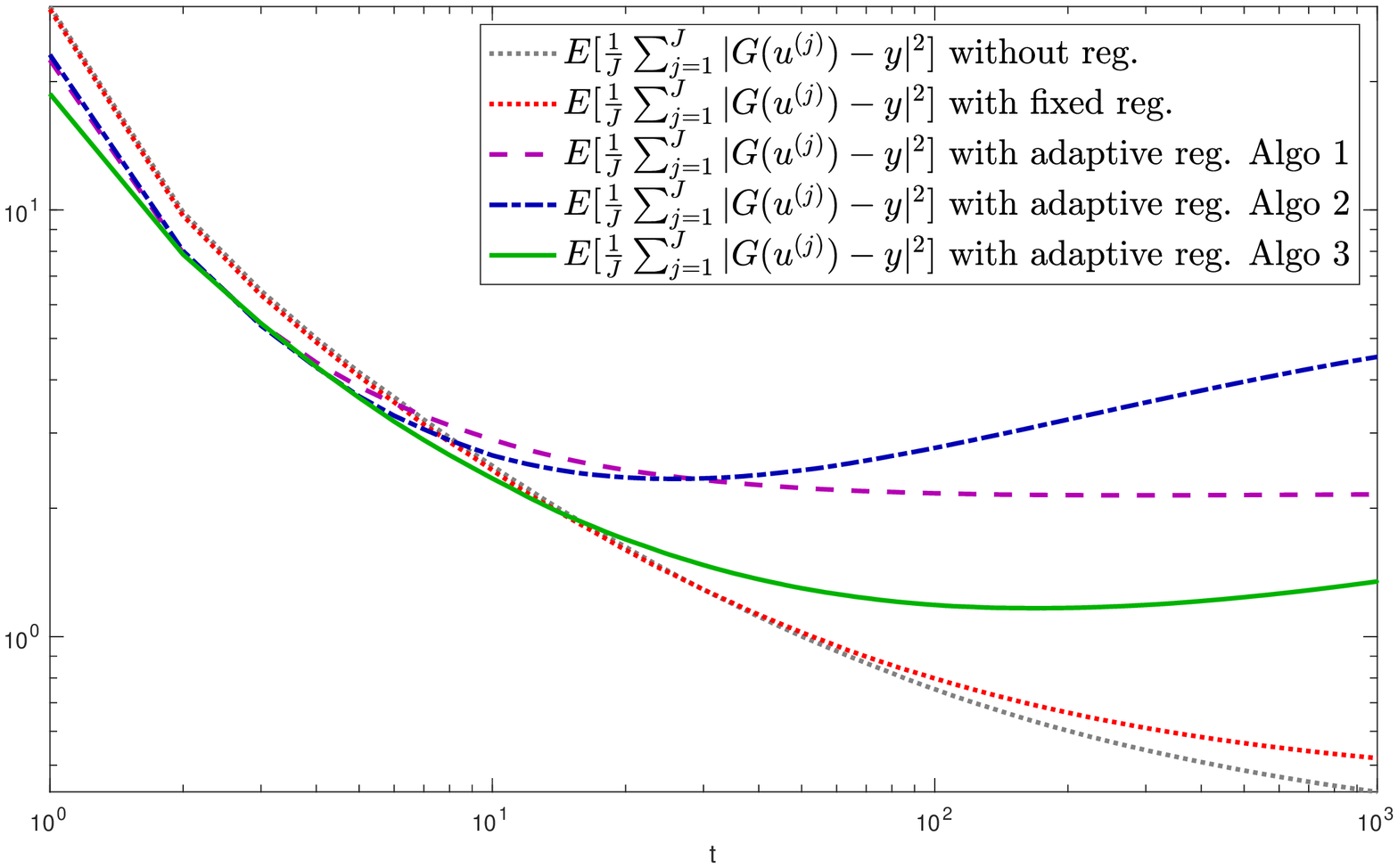}
	\end{subfigure}
	\begin{subfigure}[c]{0.49\textwidth}
	\includegraphics[width=1.1\textwidth]{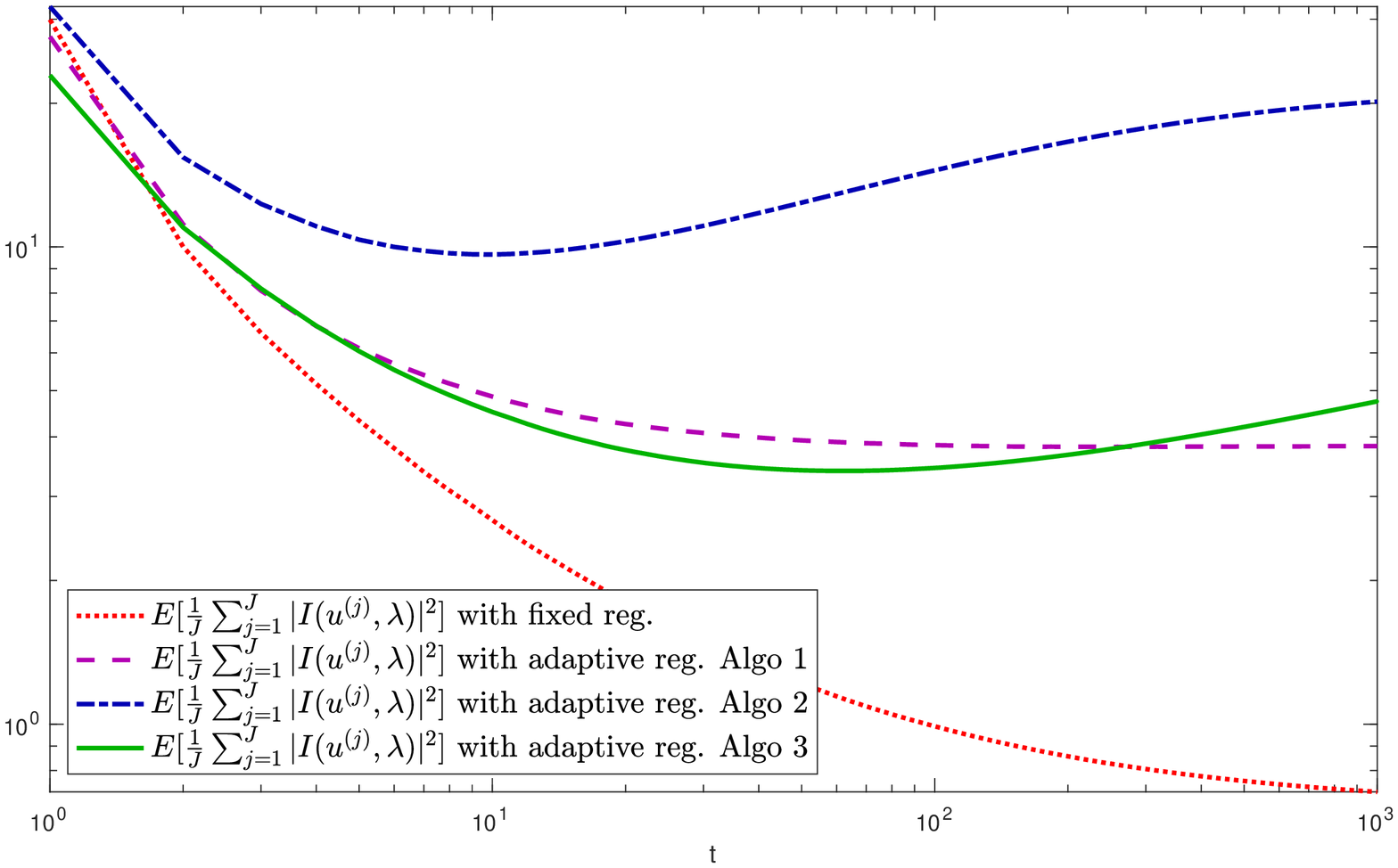}
	\end{subfigure}
     \caption{Data misfit (left) and Tikhonov regularized loss (right) for the different presented algorithms in the nonlinear example with $\lambda^\dagger = 20$. }\label{fig:nonl_ex_loss}
\end{figure} 

\begin{figure}[!htb]
	\begin{subfigure}[c]{0.49\textwidth}
	\includegraphics[width=1.1\textwidth]{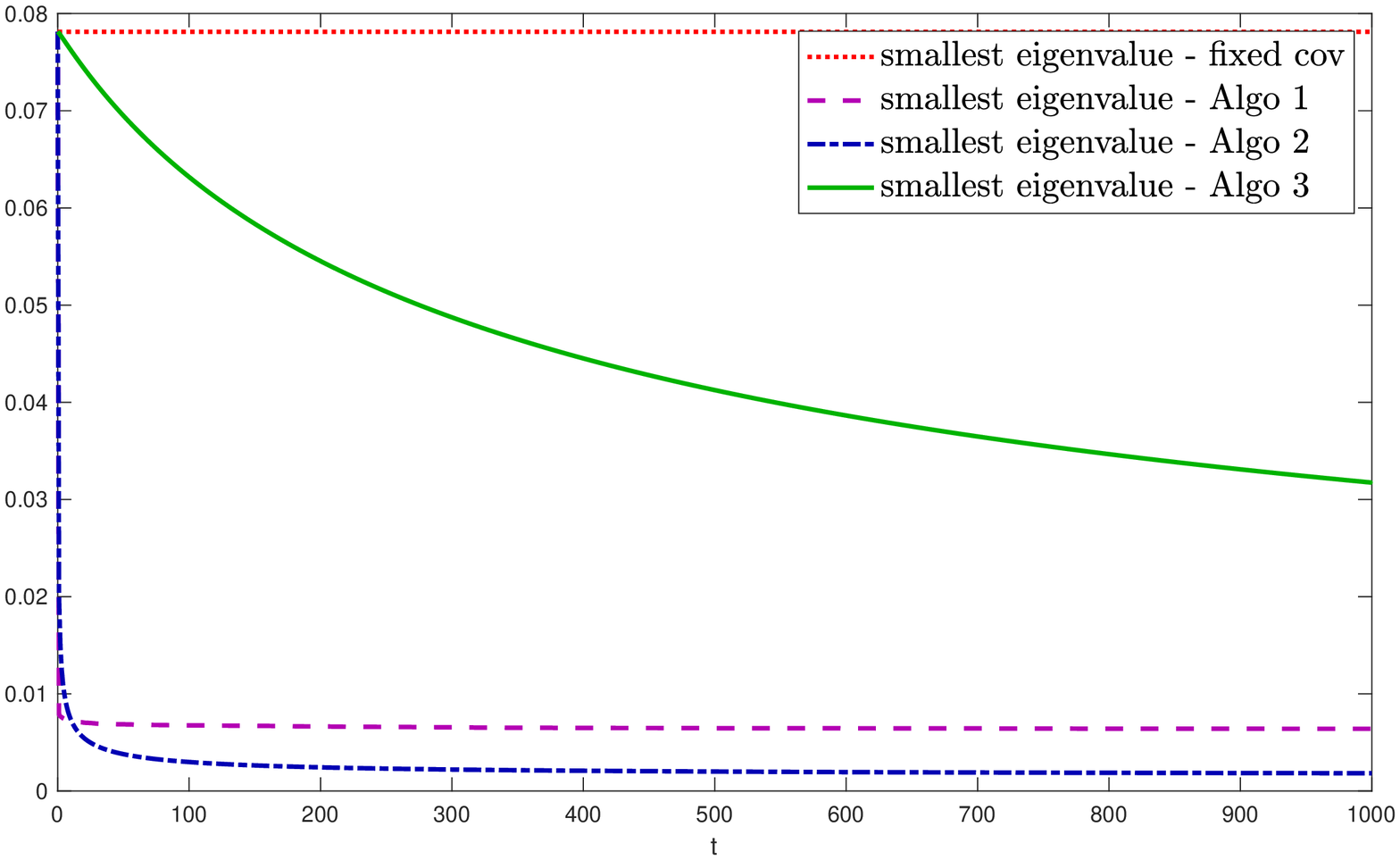}
	\end{subfigure}
	\begin{subfigure}[c]{0.49\textwidth}
	\includegraphics[width=1.1\textwidth]{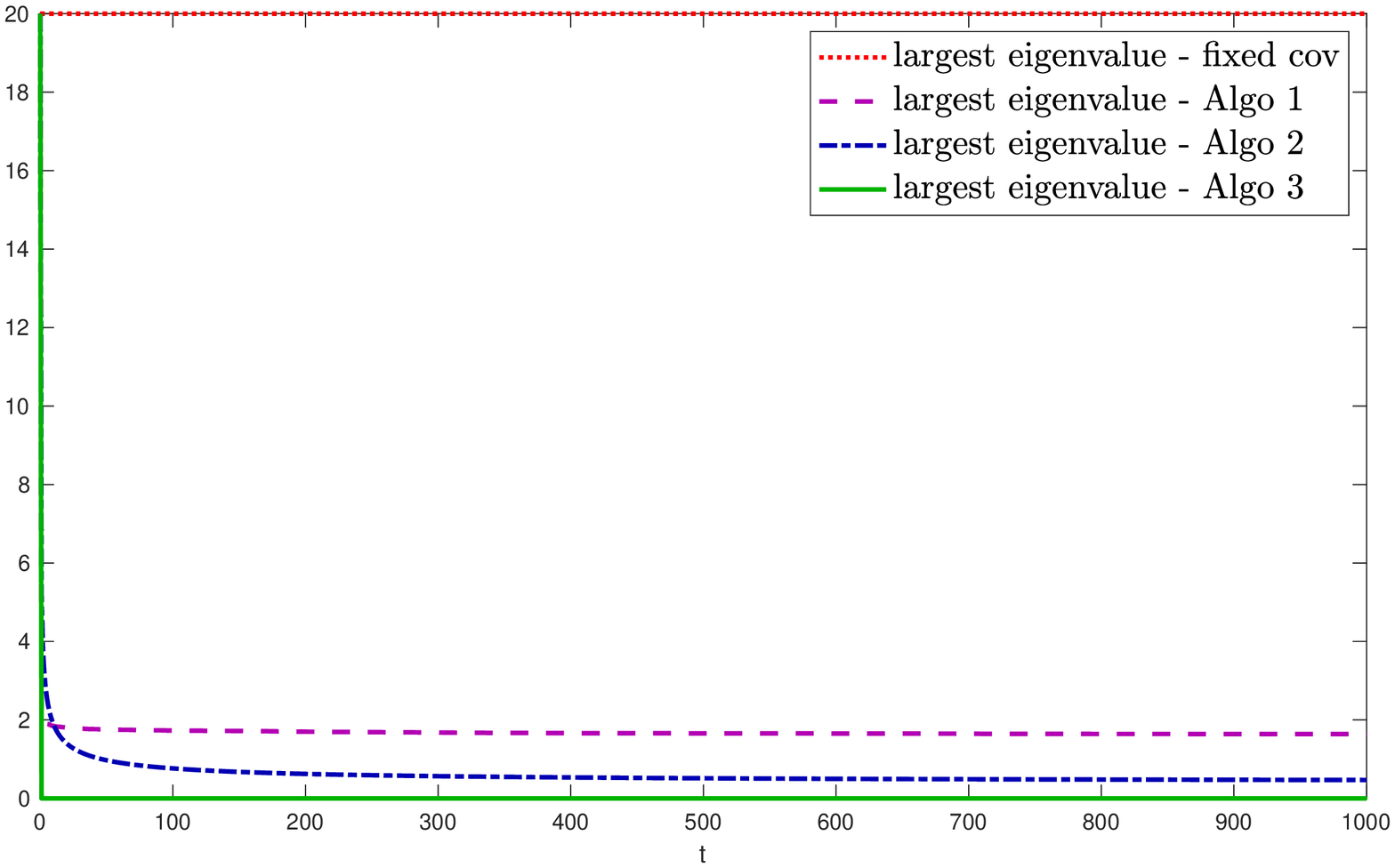}
	\end{subfigure}
    \caption{Smallest (left) and largest (right) eigenvalue of the learned regularization covariance for the different presented algorithms in the nonlinear example with $\lambda^\dagger = 20$.}\label{fig:nonl_ex_eigenvalues_cov}
\end{figure}

\section{Conclusion}
\label{sec:conc}

Regularization is an important tool in applied mathematics which can help alleviate stability issues. The motivation behind this work was to introduce adaptive regularization techniques for stochastic EKI applied to noisy measurements. We considered applying these techniques to the case when EKI is represented as a coupled system of SDEs.  Introducing the continuous-time limit under linear assumptions on the forward map we were able to provide various theoretical results, including well-posedness, the ensemble collapse and convergence of the regularized loss function, irrespective of noise. The analysis presented here is an extension of the work \cite{SS17b} which incorporates Tikhonov regularization adapted to EKI \cite{CST19}. We further introduced various adaptive schemes for choosing the regularization parameter which were tested on different models, both linear and nonlinear PDEs. From the numerical experiments it was shown that our adaptive regularization schemes outperform the fixed regularization.
 
As the theory of this work is specific to the linear case, a natural question to ask is could this be extended to the a nonlinear setting. To help achieve this one could use the tools introduced in \cite{CT19,BSWW2021}, but adapting this in a continuous-time setting. A second direction would be to introduce other forms of regularization, such as $L_1$, which has been studied in a non-adaptive fashion \cite{YL21,SSW20}.  Furthermore, we emphasize that the presented adaptive regularization schemes may also be applied to other iterative regularization methods for inverse problems such as Landweber or Levenberg-Marquardt regularization \cite{BB18,LP13}. A final potential direction could be a more complete analysis in the time-dependent setting including a time-dependent forward model as well, which is of current ongoing work of the coauthors. This could potentially allow for better inverse modeling of time-dependent PDEs, through a time-dependent EKI.

\section*{Acknowledgements} 
CS and SW are grateful to the DFG RTG1953 "Statistical Modeling of Complex Systems and Processes” for funding of this research. NKC is supported by KAUST baseline funding. XTT is supported by the National University of Singapore grant R-146-000-292-114. The authors acknowledge support by the state of Baden-W\"urttemberg through bwHPC.

\appendix
\section{Proof of Theorem \ref{thm:residuals_VI}}\label{app:1}

\begin{proof}
We first note that by chain rule we can write
\begin{equation*}
d \tilde r_t^{(j)} = d\left( \Sigma_t^{-1/2}\mathfrak r_t^{(j)})\right) = (d\Sigma_t^{-1/2}) \mathfrak r_t^{(j)} + \Sigma_t^{-1/2} d\mathfrak r_t^{(j)},
\end{equation*}
where $d\mathfrak r_t^{(j)} = Fdu_t^{(j)} - Fdu_t^\ast$ and
\begin{align*}
\frac{d u_t^\ast}{dt} &= \left(A^\top\Gamma_t^{-1}A+C_t^{-1}\right)^{-1}A^\top\Psi_t^\Gamma y\\ &\quad+ \left(A^\top\Gamma_t^{-1}A+C_t^{-1}\right)^{-1}
(A^\top\Psi_t^\Gamma A+\Psi_t^C) \left(A^\top\Gamma_t^{-1}A+C_t^{-1}\right)^{-1}\Gamma_t^{-1}y
\end{align*}
Since we assume $\left(A^\top\Gamma_t^{-1}A+C_t^{-1}\right)^{-1}\preceq C_t\preceq M_1I, \Gamma_t^{-1}\preceq M_2 I$. So
\[
{\frac{d u_t^\ast}{dt} \leq \|A\|M_1 \|\Psi_t^\Gamma\|\|y\|+M_1^2 (\|A\|^2+1)\|\Psi_t^\Gamma\|\|y\|}
\leq M_A \|\Psi_t\|.
\] 
Let $\tilde B_t ={\Sigma^{-1/2}_t}FBF^\top{\Sigma^{-1/2}_t} =: {\Sigma^{-1/2}_t}B{\Sigma^{-1/2}_t}\in\mathcal L(\R^{{d_u}+K},\R^{{d_u}+K})$ be a positive definite operator, $\alpha\in(0,1),\ R>0$ and assume, that that the smallest eigenvalues of $\tilde B_t$ are bounded by $\sigma_{\min} = c>0$. We apply chain rule and It\^o's formula in order to obtain 
\begin{align*}
d|{\Sigma^{-1/2}_t}\mathfrak r_t^{(j)}|^2 	&= 
{-2\langle{\Sigma^{-1/2}_t}\mathfrak r_t^{(j)},\Sigma^{-1/2}_t \frac{d u^*_t}{dt}dt\rangle}+{(\mathfrak r_t^{(j)})^{\top}\Psi_t \mathfrak r^{(j)}_t}\\
&\quad -2\left\langle{\Sigma^{-1/2}_t}\mathfrak r_t^{(j)},{\Sigma^{-1/2}_t}\left(C(\mathfrak r_t)+\frac{1}{t^\alpha+R}B\right){\Sigma^{-1/2}_t}{\Sigma^{-1/2}_t}\mathfrak r_t^{(j)}\right\rangle\, d t\\ &\quad+ 2\langle{\Sigma^{-1/2}_t}\mathfrak r_t^{(j)},{\Sigma^{-1/2}_t}C(\mathfrak r_t){\Sigma^{-1/2}_t}d W_t^{(j)}\rangle
\\
					&\quad+\frac1J\sum\limits_{j=1}^J\left\langle{\Sigma^{-1/2}_t}(\mathfrak r_t^{(k)}-\overline{\mathfrak r_t}),{\Sigma^{-1/2}_t}C(\mathfrak r_t){\Sigma^{-1/2}_t}{\Sigma^{-1/2}_t}(\mathfrak r_t^{(k)}-\overline{\mathfrak r_t})\right\rangle\,d t.
\end{align*}
We take the empirical mean over all particles leading to
\begin{align*}
d\frac1J\sum\limits_{j=1}^J|{\Sigma^{-1/2}_t}\mathfrak r_t^{(j)}|^2 	&= -\frac2J\sum\limits_{j=1}^J\left\langle{\Sigma^{-1/2}_t}\mathfrak r_t^{(j)},{\Sigma^{-1/2}_t}\left(C(\mathfrak r_t)+\frac{1}{t^{\alpha}+R}B\right){\Sigma^{-1/2}_t}{\Sigma^{-1/2}_t}\mathfrak r_t^{(j)}\right\rangle\,d t\\ &\quad + \frac2J\sum\limits_{j=1}^J\langle{\Sigma^{-1/2}_t}\mathfrak r_t^{(j)},{\Sigma^{-1/2}_t}C(\mathfrak r_t){\Sigma^{-1/2}_t}d W^{(j)}\rangle\\
				&\quad +\frac1J\sum\limits_{k=1}^J\langle{\Sigma^{-1/2}_t}(\mathfrak r_t^{(k)}-\overline{\mathfrak r_t}),{\Sigma^{-1/2}_t}C(\mathfrak r_t){\Sigma^{-1/2}_t}{\Sigma^{-1/2}_t}(\mathfrak r_t^{(k)}-\overline{\mathfrak r}_t)\rangle\\ &\quad +{(\mathfrak r_t^{(k)})^{\top}\Psi_t \mathfrak r^{(k)}_t}\,d t-\frac1J \sum_{j=1}^J{2\langle{\Sigma^{-1/2}_t}\mathfrak r_t^{(j)},\Sigma^{-1/2}_t \frac{d u^*_t}{dt}\rangle dt}.
\end{align*}
By applying Cauchy Schwarz inequality and the assumption $\|\frac{d u^*_t}{dt}\|\le \|\Psi_t\|$ we obtain
\[
{\langle \Sigma^{-1/2}_t\mathfrak r_t^{(j)},\Sigma^{-1/2}_t \frac{d u^*_t}{dt}\rangle\leq 
\|\Sigma^{-1/2}_t\mathfrak r_t^{(j)}\|^2\|\Psi_t\|+\|\Psi_t\|},
\]
and in fact, we have that 
\[
d\E \frac1J\sum\limits_{j=1}^J|{\Sigma^{-1/2}_t}\mathfrak r_t^{(j)}|^2
\leq t^{-\beta}(\E \frac1J\sum\limits_{j=1}^J|{\Sigma^{-1/2}_t}\mathfrak r_0^{(j)}|^2+C).
\]
Since $\beta>1$, there exists a constant $c_\beta$ such that for all $t_0\ge0$
\begin{equation*}
\E \frac1J\sum\limits_{j=1}^J|{\Sigma^{-1/2}_{t_0}} \mathfrak r_{t_0}^{(j)}|\leq c_\beta\E \frac1J\sum\limits_{j=1}^J|{\Sigma^{-1/2}_0} \mathfrak r_0^{(j)}|,
\end{equation*}
and hence, we also obtain 
\[
\E {\langle \Sigma^{-1/2}_t\mathfrak r_t^{(j)},\Sigma^{-1/2}_t \frac{d u^*_t}{dt}\rangle\leq 
 c_1\|\Psi_t\|},
\]
for some constant $c_1>0$.  With these observations, we are now able to write 
\begin{align*}
0&\le \E\left[\frac1J\sum\limits_{j=1}^J|{\Sigma^{-1/2}_t}\mathfrak r_{t+s}^{(j)}|^2\right] \\	&\le\E\left[\frac1J\sum\limits_{j=1}^J|{\Sigma^{-1/2}_s}\mathfrak r_s^{(j)}|^2\right]
+\frac{1}{J}\int^{s+t}_s \E {\langle \mathfrak r_r^{(j)},\Psi_r \mathfrak r^{(j)}_r\rangle } dr+c_1\int^{s+t}_s \|\Psi_r\| dr \\
&\quad -\frac{2}J\int_s^{s+t} \E\left[\sum\limits_{j=1}^J\langle{\Sigma^{-1/2}_t}\mathfrak r_r^{(j)},{\Sigma^{-1/2}_t}C(\mathfrak r_r){\Sigma^{-1/2}_t}{\Sigma^{-1/2}_t}\mathfrak r_r^{(j)}\rangle\right]\, d r\\
&\quad-\frac2J\int_s^{s+t} \frac{1}{r^{\alpha}+R}\E\left[\sum\limits_{j=1}^J\langle {\Sigma^{-1/2}_t}\mathfrak r_r^{(j)},{\Sigma^{-1/2}_t}B{\Sigma^{-1/2}_t}{\Sigma^{-1/2}_t}\mathfrak r_r^{(j)}\rangle\right]\, d r\\
&\quad+\frac1J\int_s^{s+t} \E\left[\sum\limits_{j=1}^J\langle{\Sigma^{-1/2}_t}(\mathfrak r_r^{(j)}-\overline{\mathfrak r}_r),{\Sigma^{-1/2}_t}C(\mathfrak r_r){\Sigma^{-1/2}_t}{\Sigma^{-1/2}_t}(\mathfrak r_r^{(j)}-\overline{\mathfrak r}_r)\rangle\right]\,d r \\
&\le \E\left[\frac1J\sum\limits_{j=1}^J|{\Sigma^{-1/2}_s}\mathfrak r_s^{(j)}|^2\right]
+\frac{1}{J}\int^{s+t}_s \E {\langle \mathfrak r_r^{(j)}),\Psi_r \mathfrak r^{(j)}_r\rangle } +c_1\|\Psi_r\| dr
 \\
 &\quad -\frac{1}J\int_s^{s+t} \E\left[\sum\limits_{j=1}^J\langle{\Sigma^{-1/2}_t}\mathfrak r_r^{(j)},{\Sigma^{-1/2}_t}C(\mathfrak r_r){\Sigma^{-1/2}_t}{\Sigma^{-1/2}_t}\mathfrak r_r^{(j)}\rangle\right]\, d r\\
&\quad-\frac2J\int_s^{s+t} \frac{1}{r^{\alpha}+R}\E\left[\sum\limits_{j=1}^J\langle {\Sigma^{-1/2}_t}\mathfrak r_r^{(j)},{\Sigma^{-1/2}_t}B{\Sigma^{-1/2}_t}{\Sigma^{-1/2}_t}\mathfrak r_r^{(j)}\rangle\right]\, d r,\\
\end{align*}
where we have used
\begin{equation*}
\sum_{k=1}^J \langle \mathfrak r^{(j)}-\bar{\mathfrak r}, C(\mathfrak r)(\mathfrak r^{(j)}-\bar{\mathfrak r})\rangle = \sum_{k=1}^J \langle \mathfrak r^{(j)} C(\mathfrak r)\mathfrak r^{(j)}\rangle - J \langle \bar{\mathfrak r}, C(\mathfrak r) \bar{\mathfrak r}\rangle \le \sum_{k=1}^J \langle \mathfrak r^{(j)}, C(\mathfrak r)\mathfrak r^{(j)}\rangle,
\end{equation*}
in the second inequality. Moreover it implies, by setting $s=t_0>1$ and using Assumption~\ref{aspt:convGamma}
\begin{align*}
c_\beta &\E\left[\frac1J\sum\limits_{j=1}^J|{\Sigma^{-1/2}_{0}}\mathfrak r_0^{(j)}|^2\right]+c_1\kappa_1 \\ &\ge \frac{1}J\int_{t_0}^{t} \E\left[\sum\limits_{j=1}^J\langle{\Sigma^{-1/2}_s}\mathfrak r_s^{(j)},{\Sigma^{-1/2}_s}(C(\mathfrak r_s)+ \frac{1}{s^{\alpha}+R}B){\Sigma^{-1/2}_s}{\Sigma^{-1/2}_s}\mathfrak r_s^{(j)}\rangle\right]\, d s\\
&\quad - \frac{1}{J}\int^{t}_{t_0} \E {\langle \mathfrak r_s^{(j)}),\Psi_s \mathfrak r^{(j)}_s\rangle }\, dr\\
&\ge \frac{1}{J} \int_{t_0}^{\top} \left(\frac{\sigma_{\min}}{s^\alpha+R}-\frac{1}{s^\beta+R}\right)\E\left[\sum\limits_{j=1}^J|{\Sigma^{-1/2}_s}\mathfrak r_s^{(j)}|^2\right]\,ds.
\end{align*}
We pick a $t_0>1$ so that $\frac{1}{s^{\beta}+R}<\frac{\sigma_{\min} }{2s^\alpha+2R}$ such that 
\begin{align*}
c_\beta &\E\left[\frac1J\sum\limits_{j=1}^J|{\Sigma^{-1/2}_{0}}\mathfrak r_0^{(j)}|^2\right]+c_1\kappa_1\\ &\ge \frac{1}{J} \int_{t_0}^{\top} \left(\frac{\sigma_{\min}}{2(s^\alpha+R)}\right)\,ds \left\{\min_{r\le t}\E\left[\sum\limits_{j=1}^J|{\Sigma^{-1/2}_r}\mathfrak r_r^{(j)}|^2\right]\right\}.
\end{align*}
We use
\begin{equation*}
\int_{t_0}^{\top} \frac{1}{2(s^\alpha+R)}\,ds \ge \frac{1}{2(1+R)(1-\alpha)}( t^{1-\alpha}-t_0),
\end{equation*}
and conclude with 
\begin{equation*}
 \left\{\min_{r\le t}\E\left[\sum\limits_{j=1}^J|{\Sigma^{-1/2}_r}\mathfrak r_r^{(j)}|^2\right]\right\} \in \mathcal O(t^{-(1-\alpha)}).
\end{equation*}

\end{proof}

\bibliographystyle{plain}
\bibliography{references.bib}

\end{document}